\documentclass{amsart}
\usepackage[T1]{fontenc}
\usepackage[utf8]{inputenc}
\usepackage{lmodern}
\usepackage{amsmath}
\usepackage{amssymb}
\usepackage{stmaryrd}
\usepackage{mathrsfs}
\usepackage{csquotes}
\usepackage[ruled]{algorithm2e}
\usepackage{tikz}
\usepackage{comment}
\usetikzlibrary{patterns}
\usepackage{empheq}
\usepackage{bbm}
\usepackage{hyperref}
\hypersetup{colorlinks=true,
            linkcolor=green!30!black,
            citecolor=cyan!50!black}

\usepackage{geometry}
\usepackage{amsthm}
\newtheorem{thm}{Theorem}[section]
\newtheorem{prop}[thm]{Proposition}
\newtheorem{lemma}[thm]{Lemma}
\newtheorem{cor}[thm]{Corollary}
\newtheorem{fact}[thm]{Fact}

\numberwithin{equation}{section}

\title{Finding minimum spanning trees via local improvements}
\address{Department of Mathematics and Statistics, McGill University, Montr\'eal, Canada}
\author{Louigi Addario-Berry}
\email{louigi.addario@mcgill.ca}
\author{Jordan Barrett}
\email{jordan.barrett@mail.mcgill.ca}
\author{Beno\^it Corsini}
\email{benoitcorsini@gmail.com}

\providecommand{\R}{}
\renewcommand{\R}{\mathbb{R}}
\newcommand{\mstr}{\ensuremath{\mathop{\mathrm{MST}}}}
\newcommand{\wt}{\ensuremath{\mathop{\mathrm{wt}}}}
\newcommand{\cost}{\ensuremath{\mathop{\mathrm{cost}}}}
\newcommand{\wdiam}{\ensuremath{\mathop{\mathrm{wdiam}}}}
\newcommand{\diam}{\mathop{\ensuremath{\mathrm{diam}}}}

\newcommand{\dist}{\mathrm{dist}}

\newcommand{\eps}{\varepsilon}
\newcommand{\rG}{\mathbb{G}}
\newcommand{\rE}{\mathrm{E}}
\newcommand{\rV}{\mathrm{V}}
\newcommand{\rU}{\mathbb{U}}
\newcommand{\rX}{\mathbb{X}}
\newcommand{\rK}{\mathbb{K}}
\newcommand{\convp}{\overset{\mathbb{P}}{\longrightarrow}}
\newcommand{\rS}{\mathbb{S}}
\newcommand{\rW}{\ensuremath{W}}
\newcommand{\rL}{\ensuremath{L}}
\newcommand{\rI}{\ensuremath{\mathrm{I}}}
\newcommand{\rA}{\ensuremath{\mathbf{A}}}

\newcommand{\Wvalue}{\ensuremath{\frac{1}{\log n}}}
\newcommand{\Lvalue}{\ensuremath{\lfloor\log\log n\rfloor}}

\newcommand{\Star}{\ensuremath{S}}
\newcommand{\Path}{\ensuremath{P}}

\definecolor{theme}{RGB}{15, 87, 24} 
\definecolor{lighttheme}{RGB}{51, 153, 63} 
\definecolor{block}{RGB}{102, 60, 0} 
\definecolor{lightblock}{RGB}{255, 238, 214} 
\definecolor{alert}{RGB}{212, 43, 43} 

\begin{document}

\date{May 10, 2022}

\begin{abstract}
    We consider a family of local search algorithms for the minimum-weight spanning tree, indexed by a parameter $\rho$. One step of the local search corresponds to replacing a connected induced subgraph of the current candidate graph whose total weight is at most $\rho$ by the minimum spanning tree (\mstr) on the same vertex set. Fix a non-negative random variable $X$, and consider this local search problem on the complete graph $K_n$ with independent $X$-distributed edge weights.
    Under rather weak conditions on the distribution of $X$, we determine a threshold value $\rho^*$ such that 
    the following holds. If the starting graph (the ``initial candidate \mstr'') is independent of the edge weights, then 
     if $\rho > \rho^*$ local search can construct the \mstr{} with high probability (tending to $1$ as $n \to \infty$), whereas if $\rho < \rho^*$ it cannot with high probability.
\end{abstract}
\maketitle

\section{Introduction}

{\em Local search} is the name for an optimization paradigm in which optimal or near-optimal solutions are sought algorithmically, via sequential improvements which are ``local'' in that at each step, the search space consists only of neighbours (in some sense) of the current solution. Well-known algorithmic examples of this paradigm include simulated annealing, hill climbing, and the Metropolis-Hasting algorithm.

A recent line of research considers the behaviour of local search on {\em smoothed} optimization problems, in which the input is either fully random or is a random perturbation of a fixed input. The goal in this setting is to characterize the running time of local search and the quality of its output. Problems approached in this vein include {\em max-cut} \cite{MR3678200,MR4294550,10.1145/3011870}, for which the allowed ``local'' improvements consist of moving a single vertex; {\em max-2CSP}  and the {\em binary function optimization problem} \cite{10.1145/3357713.3384325}, for which the allowed local improvements are bit flips; and Euclidean TSP \cite{10.1007/978-3-319-21840-3_43}, where the allowed local improvements consist of replacing edge pairs $uv,wx$ with pairs $uw,vx$ (when the result is still a tour).

In the current work, we analyze local search for the  {\em random minimum spanning tree} problem, one of the first and foundational problems in combinatorial optimization. We now briefly describe our results (for more precise statements see Section~\ref{sec:results}, below). As input to the problem, we take the randomly-weighted complete graph $\rK_n=(K_n,\rX)$, where $\rX=(X_e,e \in \rE(K_n))$ are independent copies of a random variable $X$, and an arbitrary starting graph $H_0$, which we aim to transform into the minimum-weight spanning tree $\mstr$. We fix a {\em threshold weight} $\rho > 0$; at step $k\ge 0$, a local improvement consists of choosing a connected induced subgraph of the current \mstr\ candidate $H_k$ whose current total weight is at most $\rho$, and replacing it  by the minimum weight spanning tree on the same vertex set. 

Suppose that $X$ is non-negative and has a density $f:[0,\infty) \to [0,\infty)$ which is continuous at $0$ and satisfies $f(0)>0$. 
Then writing $\rho^* = \sup\{x: \mathbb{P}(X>x)>0\}$, we prove that if $\rho > \rho^*$ then there exist local search paths which output the \mstr, whereas if $\rho < \rho^*$ then local search cannot reach the \mstr\ (and, indeed, with high probability will only achieve an approximation ratio of order $\Theta(n)$).

\subsection{Detailed statement of the results}\label{sec:results}

Let $\rG=(G,w)=(V,E,w)$ be a finite weighted connected graph, where $G=(V,E)$ is a graph and $w:E \to (0,\infty)$ are edge weights; set $\rV(\rG)=\rV(G)=V$ and $\rE(\rG)=\rE(G)=E$. For a subgraph $H$ of $G$ write $\mathop{w}(H) = \sum_{e \in \rE(H)} \mathop{w}(e)$ for its weight.
A {\em minimum spanning tree} (\mstr) of $\rG$ is a spanning tree $T$ of $G$ which minimizes $\mathop{w}(T)$ among all spanning trees of $G$.
There is a unique \mstr\ provided all edge weights are distinct; we hereafter restrict our attention to weighted graphs $\rG$ where all edge weights are distinct (and more strongly where the edge weights are linearly independent over $\R$); we call such graphs {\em generic}. For a generic weighted graph $\rG$, we write $\mstr(\rG)$ for the unique minimum spanning tree of $\rG$.

For a weighted graph $\rG=(V,E,w)$ 
and a set $S \subset V$, write $G[S]$ for the induced subgraph $G[S]=(S,E|_{S \times S})$ and $\rG[S]$ for the induced weighted subgraph $\rG[S]=(G[S],w|_{\rE(G[S])})$. Now, given a spanning subgraph $H$ of $G$, define 
$\Phi(H,S)=\Phi_\rG(H,S)$ as follows. If $H[S]$ is connected then let 
$\Phi(H,S)$ be the spanning subgraph with edge set 
$(\rE(H)\setminus \rE(H[S]))\cup \rE(\mstr(\rG[S]))$; if $H[S]$ is not connected then let $\Phi(H,S)=H$. In words, to form $\Phi(H,S)$ from $H$, we replace $H[S]$ by the minimum-weight spanning tree of $\rG[S]$, unless $H[S]$ is not connected.

Now suppose we are given a finite weighted connected graph $\rG=(V,E,w)$, a spanning subgraph $H$ of $G$, and a sequence $\rS=(S_i,1 \le i \le m)$ of subsets of $V$. 
Define a sequence of spanning subgraphs 
$(H_i,0 \le i \le m)$ as follows. 
Set $H_0=H$, and for $1 \le i \le m$ let $H_i = \Phi_\rG(H_{i-1},S_i)$. Using the previous definition of $\Phi$, this simply corresponds to sequentially replacing the subgraph of $H_{i-1}$ on $S_i$ by its corresponding minimum spanning tree (assuming $H_{i-1}$ is connected).
We refer to $\rS$ as an {\em optimizing sequence} for the pair $(\rG,H)$, 
and call $(H_i,0 \le i \le m)$ the {\em subgraph sequence corresponding to $\rS$}.
We say $\rS$ is an {\em $\mathit{MST}$ sequence} for $(\rG,H)$ if the final spanning subgraph $H_m$ is the \mstr\ of $\rG$.

The weight of step $i$ of the sequence $\rS$ is defined as 
\[
\wt(\rS,i) = \wt(\rG,H,\rS,i) := \mathop{w}\big(H_{i-1}[S_i]\big) = \sum_{e \in \rE(H_{i-1}[S_i])} \mathop{w}(e)\, ,
\]
and the weight of the whole sequence is the maximal weight of a single step:
\[
\wt(\rS) = \wt(\rG,H,\rS) := \max\Big\{\wt(\rS,i):1 \le i \le m\Big\}\, .
\]
The {\em cost} of the pair $(\rG,H)$ is defined as 
\[
\cost(\rG,H) := \min\Big\{\wt(\rS):\rS\mbox{ is an \mstr\ sequence for }(\rG,H)\Big\}\, .
\]
The following theorem is the main result of the current work. 
Write $K_n$ for the complete graph with vertex set $[n]=\{1,\ldots,n\}$, and $\rK_n=(K_n,\rX)$ for the randomly weighted complete graph, where $\rX=(X_e,e \in \rE(K_n))$ are independent Uniform$[0,1]$ random variables. 
If $\rS=(S_1,\ldots,S_m)$ is an optimizing sequence for $(\rK_n,H_n)$ then we write $H_{n,0}=H_n$ and $H_{n,i}=\Phi_{\rK_n}(H_{n,i-1},S_i)$ for $1 \le i \le m$.
Finally, we say a sequence $(E_n,n \ge 1)$ of events occurs with high probability if $\mathbb{P}(E_n) \to 1$ as $n \to \infty$.

\begin{thm}\label{thm:MAIN}
Fix any sequence $(H_n,n \ge 1)$ of connected graphs with $H_n$ being a spanning subgraph of $K_n$. Then for any $\eps > 0$, as $n \to \infty$, 
\begin{itemize}
    \item[(a)] with high probability there exists an \mstr\ sequence $\rS$ for $(\rK_n,H_n)$ with $\wt(\rS) \le 1+\eps$, and 
    \item[(b)] there exists $\delta > 0$ such that with high probability, given any optimizing sequence $\rS=(S_1,\ldots,S_m)$ for $(\rK_n,H_n)$ with $\wt(\rS) \le 1-\eps$, the final spanning subgraph $H_{n,m}$ has weight $\mathop{w}(H_{n,m}) \ge \delta n \mathop{w}(\mstr(\rK_n))$.
\end{itemize}
In particular, $\cost(\rK_n,H_n) \convp 1$ as $n \to \infty$. 
\end{thm}

We discuss possible refinements of and extensions to Theorem~\ref{thm:MAIN} in the conclusion, Section~\ref{sec:conclusion}. We also explain in that section how to extend Theorem~\ref{thm:MAIN}  to more general edge weight distributions than Uniform$[0,1]$, as described just before Section~\ref{sec:results}.

\subsection{Overview of the proof}\label{sec:overview}

In this section, we give an overview of the proof of Theorem~\ref{thm:MAIN}, while postponing the proofs of the more technical aspects to Sections~\ref{sec:eating} and~\ref{sec:clique star path} and Appendix~\ref{app:diam_bd_proof}. The lower bound of Theorem~\ref{thm:MAIN} is straightforward, so we provide it in full detail immediately.

\paragraph{Lower bound of Theorem~\ref{thm:MAIN}.}

Fix $\eps > 0$, and let $E_{n,\eps}= \{e \in \rE(H_n): X_e > 1-\eps\}$. The set $E_{n,\eps}$ is a binomial random subset of $\rE(H_n)$ in which each edge is present with probability $\eps$, so 
$\mathbb{P}(|E_{n,\eps}| \ge \eps n/2) \to 1$.

Note that, for any edge $e=uv \in \rE(H_n)\setminus \rE(\mstr(\rK_n))$, 
and any optimizing sequence $\rS=(S_1,\ldots,S_m)$ for $(\rK_n,H_n)$, if there is no set $S_i$ with $u,v \in S_i$, then $e \in H_{n,m}$. It follows that for any optimizing sequence $\rS$ with $\wt(\rS) \le 1-\eps$, the final spanning subgraph $H_{n,m}$ has $E_{n,\eps} \subset \rE(H_{n,m})$ and so on the event that $|E_{n,\eps}| \ge \eps n/2$ we have
\[
w(H_{n,m}) \ge n(1-\eps)\eps/2. 
\]
To conclude, we use that $w(\mstr(\rK_n)) \to \zeta(3)$ in probability \cite{frieze1985value}. It follows that with probability tending to $1$, both $|E_{n,\eps}| \ge \eps n/2$ and 
$w(\mstr(\rK_n))\le 2\zeta(3)$, and when both these events occur we have 
\[
w(H_{n,m}) \ge n(1-\eps)\eps/2 \ge 
w(\mstr(\rK_n)) \cdot n(1-\eps)\eps/(4\zeta(3))\, .
\]
Since this holds for any optimizing sequence with weight at most $1-\eps$, the result follows by taking $\delta=(1-\eps)\eps/(4\zeta(3))$. 

\paragraph{Upper bound of Theorem~\ref{thm:MAIN}.}

We now turn to the key ideas underlying our proof of the upper bound.
We begin with a deterministic fact.

\begin{fact}\label{fact:ramsey}
Any connected graph $H$ with vertex set $[n]$ contains an induced subgraph with at least  $\tfrac{1}{2}\sqrt{\log_2 n}$ vertices which is either a clique, a star, or a path.
\end{fact}

We prove the fact immediately since the proof is very short; but its proof can be skipped without consequence for the reader's understanding of what follows.

\begin{proof}[Proof of Fact~\ref{fact:ramsey}]
The result is trivial if $n\le 16$ so assume $n > 16$. 
Let $m=n^{1/\sqrt{\log_2 n}} \ge 4$. If $H$ has maximum degree less than $m$ then it has diameter at least $\sqrt{\log_2 n}-1 \ge \tfrac{1}{2}\sqrt{\log_2 n}$ so it contains a path of length at least $\tfrac{1}{2}\sqrt{\log_2 n}$. On the other hand, if $H$ has maximum degree at least $m$ then let $v$ be a vertex of $H$ with degree at least $m$ and let $N_v$ be the set of neighbours of $v$ in $H$. By Ramsey's theorem, and more concretely the diagonal Ramsey upper bound $R(k,k)< 4^k$, the graph $H[N_v]$ contains a set $S$ of size at least
\[
\frac{1}{2}\log_2 m = \frac{1}{2}\frac{\log_2 n}{\sqrt{\log_2 n}} = \frac{1}{2}\sqrt{\log_2 n} 
\]
such that $H[S]$ is either a clique or an independent set. If $H[S]$ is a clique then we are done, and if $H[S]$ is an independent set then $H[S \cup \{v\}]$ is a star of size $|S|+1$ so we are again done.
\end{proof}

Fact~\ref{fact:ramsey} proves to be useful together with the following special case of the upper bound of Theorem~\ref{thm:MAIN}, whose proof appears in Section~\ref{sec:clique star path}.

\begin{prop}\label{prop:MAIN}
Fix a sequence $(H_n,n \ge 1)$ of connected graphs such that, for all $n$,  $H_n$ is either a clique, a star, or a path with $\rV(H_n)=[n]$. 
Then for all $\eps > 0$, with high probability $\cost(\rK_n,H_n) \le 1+\eps$.
\end{prop}

We combine Proposition~\ref{prop:MAIN} with Fact~\ref{fact:ramsey} as follows. First, choose $V_n \subset[n]$ with $|V_n| \ge \tfrac{1}{2}\sqrt{\log n}$ such that $H_n[V_n]$ is a clique, a star or a path, and consider $\rK_n[V_n]$, the restriction of the weighted complete graph $\rK_n$ to $V_n$. Let $\rS_n'=(S_0',\ldots,S_m')$ be an \mstr\ sequence for $(\rK_n[V_n],H_n[V_n])$ of minimum cost. 
Now consider using the sequence $\rS_n'$ as an optimizing sequence for $(\rK_n,H_n)$. In other words, we set $H_{n,i}=\Phi_{\rK_n}(H_{n,i-1},S_i')$ for $1\le i \le m$. Then $H_{n,m} = \Phi_{\rK_n}(H_0,V_n)$, which is to say that $H_{n,m}$ consists of $H_n$ with $H_n[V_n]$ replaced by $\mstr(\rK_n[V_n])$. Moreover, by Proposition~\ref{prop:MAIN}, $\wt(\rS_n')=\wt(\rK_n,H_n,\rS'_n)=\wt(\rK_n[V_n],H_n[V_n],\rS'_n) \convp 1$; so with high probability we have transformed a ``large'' (i.e. whose size is $\ge \tfrac{1}{2}\sqrt{\log n}$) subgraph of $H_n$ into its minimum spanning tree, using an optimizing sequence of cost at most $1+o_{\mathbb{P}}(1)$.

The next step is to apply a procedure we call the {\em eating algorithm}, described in Section~\ref{sec:eating}. This algorithm allows us to bound the minimum cost of an \mstr\ sequence in terms of the weighted diameters of minimum spanning trees of a growing sequence of induced subgraphs of the input graph, with each graph in the sequence containing one more vertex than its predecessor. In the setting of Theorem~\ref{thm:MAIN}, it allows us to find an \mstr\ sequence with weight at most $1+o_{\mathbb{P}}(1)$ provided that the starting graph already contains a large subgraph on which it is equal to the \mstr. The key result of our analysis of the eating algorithm is summarized in the following proposition.

\begin{prop}\label{prop:eating}
Fix a sequence $(H_n,n \ge 1)$ of connected graphs with $\rV(H_n)=[n]$. 
Fix any sequence of sets $(V_n,n \ge 1)$ such that $V_n \subset [n]$, $|V_n| \to \infty$ as $n \to \infty$, and $H_n[V_n]$ is connected for all $n \ge 1$. Let  $H_n'=\Phi_{\rK_n}(H_n,V_n)$, so that  $H_n'[V_n] = \mstr(\rK_n[V_n])$. Then for all $\eps > 0$, 
with high probability $\cost(\rK_n,H_n')\le 1+\eps$.
\end{prop}

The proof of Proposition~\ref{prop:eating} appears in Section~\ref{sec:eating}. We are now prepared to prove Theorem~\ref{thm:MAIN}, modulo the proofs of Proposition~\ref{prop:MAIN} and Proposition~\ref{prop:eating}.

\begin{proof}[Proof of Theorem~\ref{thm:MAIN}]
We already established the lower bound of the theorem, so it remains to show that for all $\eps > 0$, 
\[
\mathbb{P}\big(\cost(\rK_n,H_n) \leq 1+\eps\big) \longrightarrow 1
\]
as $n \to \infty$. For the remainder of the proof we fix $\eps > 0$.

Using Fact~\ref{fact:ramsey}, let $V_n$ be a subset of $[n]$ with size at least $\tfrac{1}{2}\sqrt{\log_2 n}$ such that $H_n[V_n]$ is a clique, a star or a path.
Write $\rK_n^-=\rK_n[V_n]$ and $H_n^-=H_n[V_n]$, and let $\rS^-_n$ be an \mstr\ sequence for $(\rK_n^-,H_n^-)$ of minimum cost. 
By Proposition~\ref{prop:MAIN}, 
\[
\mathbb{P}\Big(\wt(\rK_n^-,H_n^-,\rS^-_n) \le 1+\eps\Big) \longrightarrow 1
\]
as $n \to \infty$. Moreover, we have $\wt(\rK_n^-,H_n^-,\rS^-_n) = \wt(\rK_n,H_n,\rS^-_n)$: 
the weight of the sequence $\rS_n^-$ is the same with respect to $(\rK_n^-,H_n^-)=(\rK_n[V_n],H_n[V_n])$ as it is with respect to $(\rK_n,H_n)$;
this is easily seen be induction.
It follows that 
\[
\mathbb{P}\Big(\wt(\rK_n,H_n,\rS^-_n) \leq 1+\eps\Big) \longrightarrow 1\, .
\]

Next let $H_n'=\Phi_{\rK_n}(H_n,V_n)$, so $H_n'[V_n]=\mstr(\rK_n[V_n])$. Since $\rS_n-$ is an \mstr\ sequence for $(\rK_n^-,H_n^-)$, this is also the graph resulting from using $\rS^-_n$ as an optimizing sequence for $(\rK_n,H_n)$. Now let $\rS_n'$ be an \mstr\ sequence for $H_n'$ of minimum cost. Since $|V_n| \to \infty$ and $H_n[V_n]$ is connected, it follows from Proposition~\ref{prop:eating} that 
\[
\mathbb{P}\Big(\wt(\rK_n,H_n',\rS_n') \leq 1+\eps\Big) \longrightarrow 1\, .
\]

To conclude, note that the concatenation $\rS_n$ of $\rS^-_n$ and $\rS_n'$ is an \mstr\ sequence for $(\rK_n,H_n)$, and 
\[
\wt(\rK_n,H_n,\rS_n) = \max\Big\{\wt(\rK_n,H_n,\rS^-_n),\wt(\rK_n,H_n',\rS_n')\Big\}\, ,
\]
so $\mathbb{P}(\wt(\rK_n,H_n,\rS_n) \leq 1+\eps)\to 1$ and thus $\mathbb{P}(\cost(\rK_n,H_n) \leq 1+\eps) \to 1$, as required. 
\end{proof}

The remainder of the paper proceeds as follows. In Section~\ref{sec:eating} we describe the eating algorithm and prove Proposition~\ref{prop:eating}, modulo the proof of a key technical input (Theorem~\ref{thm:bound on diam(mst)}), an upper tail bound on the weighted diameter of $\mstr(\rK_n)$, which is postponed to Appendix~\ref{app:diam_bd_proof}. In Section~\ref{sec:clique star path} we prove Proposition~\ref{prop:MAIN} by using the details of the eating algorithm to generate a well bounded sequence of increasing \mstr{}s that are each built from a clique, a star, or a path. We conclude in Section~\ref{sec:conclusion} by presenting the generalization of Theorem~\ref{thm:MAIN} to other edge weight distributions, and by discussing avenues for future research. 

\section{The eating algorithm}\label{sec:eating}

In this section, we prove Proposition~\ref{prop:eating}.
Informally, we prove this proposition by showing that we can efficiently add vertices to an MST of a large subgraph of $K_n$, one at a time, via an optimizing sequence which has a low weight, with high probability.
For a weighted graph $\rG=(V,E,w)$, write $\wdiam(\rG)$ for the weighted diameter of $\rG$,
\[ 
\wdiam(\rG):= \max\Big\{\dist_\rG(u,v):u,v\in V\Big\}\,,
\]
where
\[
    \dist_\rG(u,v):=\min\Big\{\mathop{w}(P):P\mbox{ is a path from $u$ to $v$ in }\rG\Big\}\,.
\]
It is sometimes convenient to write $\wdiam(G)$ for an unweighted graph $G$, where the appropriate choice of weights is clear from context.
Finally, we also introduce the unweighted diameter
\begin{align*}
\diam(G) := \max\Big\{\min\Big\{\big|\rE(P)\big|:\textrm{$P$ is a path from $u$ to $v$ in $G$}\Big\}:u,v\in V\Big\}\,,
\end{align*}
which will be used later in this work (in Section~\ref{sec:good sets star path} and in Appendix~\ref{app:diam_bd_proof}).

The key tool to prove Proposition~\ref{prop:eating} is the following proposition, which will be applied recursively.

\begin{prop}\label{prop:eating algorithm}
Let $\rG=(V,E,w)$ be a generic weighted graph with $V=[n]$ and $\max\{\mathop{w}(e):e \in E)\}\leq1$. Suppose that
$H$ is a spanning subgraph of $G$ and $H[n-1] = \mstr(\rG[n-1])$. 
Then
\begin{align*}
    \cost(\rG,H) \le 1 + \max\Big\{\wdiam\big(\mstr(\rG[n-1])\big),\wdiam\big(\mstr(\rG)\big)\Big\}\,.
\end{align*}
\end{prop}

The proof of Proposition~\ref{prop:eating algorithm} occupies the bulk of Section~\ref{sec:eating}; it appears below in Sections~\ref{sec:eating_specialcase} and~\ref{sec:eating_generalcase}.

\begin{cor}[The eating algorithm]\label{cor:eating algorithm}
    Let $\rG=(V,E,w)$ be a weighted graph with $V=[n]$ and $\max\{\mathop{w}(e):e\in E\}\leq1$. Let $H$ be a spanning subgraph of $G$ and fix a non-empty set $U\subset[n]$ for which $H[U]=\mstr(\rG[U])$. Let $U=U_0\subset U_1\subset\ldots\subset U_k=V$ be any increasing sequence of subsets of $V$ such that, for all $0\leq i<k$, $U_{i+1}\setminus U_i$ is a singleton and $H[U_i]$ is connected. Then
    \begin{align*}
        \cost(\rG,H)\leq 1+\max\Big\{\wdiam\big(\mstr(\rG[U_i])\big):0\leq i\leq k\Big\}\,.
    \end{align*}
\end{cor}

\begin{proof}
    Set $F_0=H$ and let $\rS_1,\ldots,\rS_k$ and $F_1,\ldots,F_k$ be constructed inductively as follows. Given $F_{i-1}$, let $\rS_i$ be an \mstr\ sequence of minimal weight for the pair $(\rG[U_i],F_{i-1}[U_i])$ and let $F_i=\Phi_{\rG}(F_{i-1},U_i)$. Note that $F_i[U_i]$ is the last graph of the subgraph sequence corresponding to $\rS_i$. 
    
    By using that an optimizing sequence on $(\rG[U_i],F_{i-1}[U_i])$ can also be seen as an optimizing sequence on $(\rG,F_{i-1})$ of identical weight, we can bound the weight of the global optimizing sequence $\rS$ obtained by concatenating $\rS_1,\ldots,\rS_k$ in that order. Indeed, we have that
    \begin{align*}
        \wt(\rG,H,\rS)=\max\Big\{\wt(\rG[U_i],F_{i-1}[U_i],\rS_i):1\leq i\leq k\Big\}\,.
    \end{align*}
    Moreover, by the definition of $F_{i-1}$, we know that $F_{i-1}[U_{i-1}]=\mstr(\rG[U_{i-1}])$ and by minimality of $\rS_i$ along with Proposition~\ref{prop:eating algorithm}, it follows that for all $1 \le i \le k$, 
    \begin{align*}
        \wt(\rG[U_i],F_{i-1}[U_i],\rS_i)\leq1+\max\Big\{\wdiam\big(\mstr(\rG[U_{i-1}])\big),\wdiam\big(\mstr(\rG[U_i])\big)\Big\}\,.
    \end{align*}
    Since $\cost(\rG,H) \le \wt(\rG,H,\rS)$, 
    combining the last two results provides us with the desired upper bound for $\cost(\rG,H)$.
\end{proof}

The importance of this corollary becomes clear in light of the next theorem, which provides strong tail bounds on the diameter of \mstr s of randomly-weighted complete graphs.

\begin{thm}\label{thm:bound on diam(mst)}
Let $\rK_n=(K_n,\rX)$ be the complete graph with vertex set $[n]$, endowed with independent, Uniform$[0,1]$ edge weights $\rX=(X_e,e \in \rE(K_n))$. 
Then for all $n$ sufficiently large, 
\[
\mathbb{P}\left(\wdiam\big(\mstr(\rK_n)\big) \geq \frac{7 \log^4 n}{n^{1/10}}\right) \le \frac{4}{n^{\log n}}\, .
\]
In particular, $\wdiam(\mstr(\rK_n)) \stackrel{\mathbb{P}}{\longrightarrow} 0$ as $n \to \infty$. 
\end{thm}
The proof of Theorem~\ref{thm:bound on diam(mst)} is postponed to Appendix~\ref{app:diam_bd_proof}. 
We now use Corollary~\ref{cor:eating algorithm} and Theorem~\ref{thm:bound on diam(mst)} to prove Proposition~\ref{prop:eating}.

\begin{proof}[Proof of Proposition~\ref{prop:eating}]
Consider any sequence of sets $(V_n,n\geq1)$ with $V_n\subset[n]$ and $|V_n|\to\infty$ as $n \to \infty$ and such that $H_n[V_n]$ is connected for all $n \geq 1$, and let $H'_n=\Phi_{\rK_n}(H_n,V_n)$. Since $H'_n$ is connected, we may list the vertices of $[n]\setminus V_n$ as $v_1,\ldots,v_k$ so that for all $1 \le i \le k$, vertex $v_i$ is adjacent to an element of $V_n \cup\{v_1,\ldots,v_{i-1}\}$. Taking $U_0=V_n$ and $U_i=V_n \cup\{v_1,\ldots,v_i\}$ for $1 \le i \le k$, the sequence $U_0,\ldots,U_k$ satisfies the conditions of Corollary~\ref{cor:eating algorithm} with $\rG=\rK_n$. It follows that 
 \begin{equation}\label{eq:cost_bound}
\cost(\rK_n,H'_n)\leq 1+\max\Big\{\wdiam\big(\mstr(\rK_n[U_i])\big):0\leq i\leq k\Big\}\,.
\end{equation}
Moreover, since $|V_n| \to \infty$ as $n \to \infty$, for $n$ sufficiently large we may apply Theorem~\ref{thm:bound on diam(mst)} to $\rK_n[U_i]$ for each $0 \le i \le k$ and obtain that
\begin{align*}
    \mathbb{P}\Big(\exists i:\wdiam\big(\mstr(\rK_n[U_i])\big)\geq |U_i|^{-\frac{1}{11}}\Big)&\leq\sum_{i=0}^k\mathbb{P}\Big(\wdiam\big(\mstr(\rK_n[U_i])\big)\geq |U_i|^{-\frac{1}{11}}\Big)
    \le\sum_{i=0}^k\frac{1}{|U_i|^2}\,;
\end{align*}
where we have used that $\tfrac{7 \log^4 n}{n^{1/10}} \le \tfrac{1}{n^{1/11}}$ and that $\tfrac{4}{n^{\log n}} < \tfrac{1}{n^2}$ for $n$ large. 
Since $|U_i|=|U_0|+i=|V_n|+i$, it follows that for all $n$ sufficiently large, 
\begin{align*}
    \mathbb{P}\Big(\exists i:\wdiam\big(\mstr(\rK_n[U_i])\big)\geq |U_i|^{-\frac{1}{11}}\Big)&\leq \sum_{s=|V_n|}^{n}\frac{1}{s^2}\le \frac{1}{|V_n|-1}\longrightarrow 0\,.
\end{align*}
In view of \eqref{eq:cost_bound}, this yields that 
\begin{align*}
    \mathbb{P}\Big(\cost(\rK_n,H'_n)\geq 1+\epsilon\Big)\leq\mathbb{P}\Big(\exists i:\wdiam\big(\mstr(\rK_n[U_i])\big)\geq\epsilon\Big)\longrightarrow0\,,
\end{align*}
as desired.
\end{proof}
The remainder of Section~\ref{sec:eating} is devoted to proving Proposition~\ref{prop:eating algorithm}.

\subsection{A special case of Proposition~\ref{prop:eating algorithm}}\label{sec:eating_specialcase}
To prove Proposition~\ref{prop:eating algorithm}, we need to bound $\cost(\rG,H)$ when $H$ is a spanning subgraph of $G$ with $H[n-1]=\mstr(\rG[n-1])$. It is useful to first treat the special case that $H$ only contains one edge which does not lie in $\mstr(\rG[n-1])$, and more specifically 
that $n$ is a leaf and $H$ is a tree. We will later use this case as an input to the general argument.

\begin{prop}\label{prop:eating_simple}
In the setting of Proposition~\ref{prop:eating algorithm}, if $n$ is a leaf of $H$ then 
\begin{align*}
    \cost(\rG,H) \le 1 + \wdiam\big(\mstr(\rG[n-1])\big)\, .
\end{align*}
\end{prop}

The next lemma will be useful in the proof of both the special case and the general case; informally, it states that optimizing sequences never remove MST edges that are already present, and that optimizing sequences do not create cycles. 

\begin{lemma} \label{lem:usefulfacts}
Let $(S_i, 1 \leq i \leq m)$ be an $\mstr$ sequence for $(\rG,H)$ with corresponding spanning subgraph sequence $(H_i, 0 \leq i \leq m)$. Then 
\begin{enumerate}
\item if $e \in \rE(\mstr(\rG))$ and $e \in \rE(H_i)$, then $e \in \rE(H_j)$ for all $i \leq j \leq m$, and
\item if $H_i$ is a tree, then $H_j$ is a tree for all $i \leq j \leq m$.
\end{enumerate}
\end{lemma}

\begin{proof}
We use the standard fact that if $\rG = (V,E,w)$ is a weighted graph with all edge weights distinct, then $e \in \rE( \mstr(\rG) )$ if and only if $e$ is not the heaviest edge of any cycle in $\rG$. 

Fix $e \in \rE(\mstr(\rG))$ and suppose that $e \in \rE(H_i)$. If the endpoints of $e$ do not both lie in $S_{i+1}$ then clearly $e \in \rE(H_{i+1})$ since $H_i$ and $H_{i+1}$ agree except on $S_{i+1}$. If the endpoints of $e$ both lie in $S_{i+1}$ then since $e$ is not the heaviest edge of any cycle in $\rG$, it is not the heaviest edge of any cycle in $\rG[S_{i+1}]$. Thus $e \in \rE(\mstr(\rG[S_{i+1}]))$, and so again $e \in \rE(H_{i+1})$. It follows by induction that $e \in \rE(H_j)$ for all $i \le j \le m$. 

The second claim of the lemma is immediate from the the fact that if $T$ is any tree, 
$S$ is a subset of $\rV(T)$ such that $T[S]$ is a tree, and $T'$ is another tree with $\rV(T)=S$, then the graph with vertices $\rV(T)$ and edges $(\rE(T)\setminus \rE(T[S]))\cup \rE(T')$ is again a tree. 
\end{proof}

We now assume $\rG$ and $H$ are as in Proposition~\ref{prop:eating_simple}. 
Define an optimizing sequence 
$\rS = (S_i, 1 \leq i \leq n-1)$ 
for $(\rG,H)$ as follows. Let $S_1$ be the set of vertices on the path from $n$ to $1$ in $H_0=H$, and let $H_1=\Phi_\rG(H_0,S_1)$. 
Then, inductively, for $1< i \le n-1$ let $S_i$ be the set of vertices on the path from $n$ to $i$ in $H_{i-1}$ and let $H_i=\Phi_\rG(H_{i-1},S_i)$. Since $H=H_0$ is a tree, by point 2 of Lemma~\ref{lem:usefulfacts} it follows that $H_i$ is a tree for all $i$, so the paths $S_i$ are uniquely determined and the sequence $\rS$ is well-defined.

Proposition~\ref{prop:eating_simple} is now an immediate consequence of the following two lemmas.
\begin{lemma} \label{lem:simplesequence}
$\rS$ is an $\mstr$ sequence for $(\rG,H)$.
\end{lemma}

\begin{proof}
Since $H_m$ is a tree, it suffices to show that $\mstr(\rG)$ is a subtree of $H_m$.
Let $e \in \rE(\mstr(\rG))$. Then either $e \in \rE(H_0[n-1])$ or $e = in$ for some $i \in [n-1]$. If $e \in \rE(H_0[n-1])$ then $e \in \rE(H_0)$ meaning that, by point 1 of Lemma~\ref{lem:usefulfacts}, we have $e \in \rE(H_m)$. Otherwise, if $e = in$ for some $i \in [n-1]$, then $e \in \rE(\rG[S_i])$ since $S_i$ is the set of vertices on a path from $n$ to $i$. Hence, $e \in \rE(\mstr(\rG[S_i]))$, meaning that $e \in \rE(H_i)$. Once again, by point 1 of Lemma~\ref{lem:usefulfacts}, this implies that $e\in\rE(H_m)$, proving that $\mstr(\rG)$ is a subtree of $H_m$.
\end{proof}

\begin{lemma} \label{lem:simple_weightbound}
$\wt(\rS) \leq 1 + \wdiam(\mstr(\rG[n-1])$
\end{lemma}

\begin{proof}
Let $i \in [n-1]$. Notice that the path from $n$ to $i$ in $H_{i-1}$ contains a single edge from $n$ to $[n-1]$. Hence, the weight of this path is bounded from above by $1+\wdiam(H_{i-1}[n-1])$. To prove the lemma it therefore suffices to show that $\rE(H_{i}[n-1]) \subseteq \rE(H_0[n-1]) =\rE( \mstr(\rG[n-1]))$. 

We prove this by induction on $i$, the base case $i=0$ being automatic. For $i > 0$, suppose that $\rE(H_{i-1}[n-1])\subseteq \rE(H_0[n-1])$. Fix any vertices $u,v \in S_i\cap [n-1]$ with $uv \not\in \rE(H_{i-1})$ and let $P$ be the path from $u$ to $v$ in $H_{i-1}$. Then $P$ is a subpath of $H_{i-1}[S_i]$, and so by induction it is also a subpath of $H_0$. Since $H_0[n-1]=\mstr(\rG[n-1])$ it follows that $P$ is a subpath of $\mstr(\rG[n-1])$. This yields that $uv$ is the edge with highest weight on the cycle created by closing $P$, and all the vertices of this cycle lie in $S_i$; so $uv \not\in \rE(\mstr(\rG[S_i]))$ and thus $uv \not\in \rE(H_i)$. 
This shows that $\rE(H_i[S_i]) \subseteq \rE(H_{i-1}[S_i]) \subseteq \rE(H_0[S_i])$. 
 Since the rest of $H_{i-1}[n-1]$ and $H_i[n-1]$ are identical, it follows that $\rE(H_i[n-1])\subseteq \rE(H_0[n-1])$, as required.
\end{proof}

\subsection{The general case of Proposition~\ref{prop:eating algorithm}}\label{sec:eating_generalcase}

We now lift the assumption that $H$ is a tree; in this case, $\rE(H) \setminus \rE(H[n-1])$ could contain up to $n-1$ edges. As a result, the $\mstr$ sequence previously defined in Section~\ref{sec:eating_specialcase} does not provide us with the desired cost, since a path from $n$ to $i \in [n-1]$ might contain additional edges with $n$ as an endpoint, increasing the weight of the sequence. Thus, we require a more careful method. Informally, our approach is to first apply the method from the previous section to a sequence of subgraphs of $H[n-1]$, each of which is only joined to the vertex $n$ by a single edge, but together which contain all the edges from $n$ to $[n-1]$. We show that this yields a graph which contains the \mstr\ of $G$. We then prove that any cycles in the resulting graph can be removed at a low cost.

Let $\rG = (V,E,w)$ be a generic weighted graph with $V = [n]$ and let $H$ be a spanning subgraph of $G$ with $H[n-1] = \mstr(\rG[n-1])$. Let $\{v_1n,\dots,v_kn\} \subseteq \rE(H)$ be the set of edges in $H$ with $n$ as an endpoint, and for $1 \le i \le k$ let 
\[
V_i = \bigg\{v \in [n-1]: \dist_{(H[n-1],w)}(v_i,v) = \min\Big\{\dist_{(H[n-1],w)}(v_i,v_j):1 \le j \le k\Big\}\bigg\}. 
\]
That is to say, $(V_i,1 \le i \le k)$ is the Voronoi partition of $[n-1]$ in $H[n-1]$ with respect to the vertices $v_1,\ldots,v_k$; it is indeed a partition since $\rG$ is generic.

Note that since $H[n-1]=\mstr(\rG[n-1])$ it follows that $H[V_i]=\mstr(\rG[V_i])$ for any $1\leq i\leq k$. Moreover, vertex $n$ has degree one in $H[V_i\cup \{n\}]$. Using Proposition~\ref{prop:eating_simple}, let $\rS_i=(S_{i,j},1\leq j\leq m_i)$ be an \mstr\ sequence for $(\rG[V_i\cup\{n\}],H[V_i\cup\{n\}])$ with weight less than $1+\wdiam(\mstr(\rG[V_i]))\leq 1+\wdiam(\mstr(\rG[n-1]))$, and write $(H_{i,j},0\leq j\leq m_i)$ for the corresponding subgraph sequence.
Now set $m=m_1+\ldots+m_k$ and let $\rS^*=(S^*_1,\ldots,S^*_m)$ be formed by concatenating $\rS_1,\ldots,\rS_m$, so 
\[
\rS^* = (S_{1,1},\ldots,S_{1,m_1},\ldots,S_{k,1},\ldots,S_{k,m_k})\, ,
\]
and let $(H_0^*,\ldots,H_m^*)$ be the subgraph sequence corresponding to $\rS^*$. 

\begin{lemma} \label{lem:step1_general}
We have $\mstr(\rG) \subseteq H_m^*$, and 
$\wt(\rS^*) \le 1 + \diam(\mstr(\rG[n-1]))$. 
\end{lemma}
\begin{proof}
First, by assumption, $H_0[n-1] = \mstr(\rG[n-1])$. Since $\mstr(\rG)[n-1]$ is a subgraph of $\mstr(\rG[n-1])$, point 1 of Lemma~\ref{lem:usefulfacts} implies that $\mstr(\rG)[n-1]$ is a subgraph of $H^*_i$ for all $i$, so in particular of $H^*_m$. 

Next, since $V_1,\ldots,V_k$ are disjoint, we have $S_{i,j} \cap S_{i',j'} = \{n\}$ whenever $i \ne i'$, and it follows that $H^*_{m_1+\ldots+m_{i-1}}[V_{i}\cup\{n\}] = H[V_i \cup \{n\}]$ for all $1 \le i \le k$. This implies that 
$H^*_{m_1+\ldots+m_{i-1}+j}[V_{i}\cup\{n\}]=H_{i,j}$ for each $1 \le j \le m_i$, so in particular $H^*_{m_1+\ldots+m_i}[V_i\cup\{n\}]=\mstr(\rG[V_i \cup \{n\}])$.

Now fix any edge $vn$ of $\mstr(\rG)$. Then $v \in V_i$ for some $1 \le i \le k$, so $vn \in\rE(\mstr(\rG[V_{i}\cup\{n\}]))$. It follows that $vn \in H^*_{m_1+\ldots+m_i}$, and thus by point 1 of Lemma~\ref{lem:usefulfacts} that $vn$ is an edge of $H^*_m$. Therefore all edges of $\mstr(\rG)$ are edges of $H_m^*$, as required. 

Finally, the bound on the weight of the sequence is immediate by the definition of $\rS^*$ and by using that $\wt(\rG[V_i\cup\{n\}],H[V_i\cup\{n\}],\rS_i)=\wt(\rG,H,\rS_i)$.
\end{proof}

We are now left to deal with the edges $\rE(H^*_m) \setminus \rE(\mstr(\rG))$. This is taken care of in the following lemma.
\begin{lemma}\label{lem:step2_general}
Let $\rG=(V,E,w)$ be a generic weighted graph with $V=[n]$ and with all edge weights at most $1$, and let $H$ be a subgraph of $G$ such that $\mstr(\rG)$ is a subgraph of $H$. Write $k=|\rE(H)|-(n-1)$. Then there exists an \mstr\ sequence $\rS'=(S'_1,\ldots,S'_k)$ with 
\[
\wt(\rS') \le 1 + \wdiam\big(\mstr(\rG)\big)\, .
\]
\end{lemma}

\begin{proof}
If $H$ is a tree then there is nothing to prove, so assume $G$ contains at least one cycle (so $k \ge 1$). In this case there exist vertices $u,v$ which are not adjacent in $\mstr(\rG)$ but are joined by an edge in $H$; choose such $u$ and $v$ so that the length (number of edges) on the path $P$ from $u$ to $v$ in $\mstr(\rG)$ is as small as possible. Let $S=\rV(P)$ be the set of vertices of the path $P$; then $H[S]$ is a cycle (by the minimality of the length of $P$), and $uv$ is the edge with largest weight on $H[S]$. It follows that $\mstr(\rG[S])=P$, so 
$\Phi_\rG(H,S)$ has edge set $E=\rE(H)\setminus\{uv\}$. Moreover, since $P$ is a path of $\mstr(\rG)$, it follows that  \[
\mathop{w}(H[S])=\mathop{w}(uv)+\wt(P) \le 1+\wdiam(\mstr(\rG))\, .
\]
Since $\Phi_\rG(H,S)$ contains $\mstr(G)$ but has one fewer edge than $H$, the result follows by induction.
\end{proof}

We now combine Lemmas~\ref{lem:step1_general} and \ref{lem:step2_general} to conclude the proof of Proposition~\ref{prop:eating algorithm}.

\begin{proof}[Proof of Proposition~\ref{prop:eating algorithm}]
Let $\rS^*=(S_1^*,\ldots,S_m^*)$ be the optimization sequence defined above Lemma~\ref{lem:step1_general}, and let $(H_0^*,\ldots,H_m^*)$ be the corresponding subgraph sequence. By that lemma, $\mstr(\rG)$ is a subgraph of $H_m^*$ and $\wt(\rS^*) \le 1 +\wdiam(\mstr(\rG[n-1]))$. 

Next let $\rS'=(S_1',\ldots,S_k')$ be an \mstr\ sequence for $(\rG,H_m^*)$ of weight at most $1+\wdiam(\mstr(\rG))$; the existence of such a sequence is guaranteed by Lemma~\ref{lem:step2_general}. Then the concatenation 
\[
\rS=(S_1^*,\ldots,S_m^*,S_1',\ldots,S_k')
\]
of $\rS^*$ and $\rS'$ is an \mstr\ sequence for $(\rG,H)$, of weight at most 
\[
\wt(\rG,H,\rS)\leq 1+ \max\Big\{\wdiam\big(\mstr(\rG[n-1])\big),\wdiam\big(\mstr(\rG)\big)\Big\},
\]
and the desired bound on $\cost(\rG,H)$ follows. 
\end{proof}

\section{MST sequences for the the clique, the star, and the path}\label{sec:clique star path}

This section is aimed at proving Proposition~\ref{prop:MAIN}. We start by proving the result in the case of the clique, since it is straightforward using the result of Lemma~\ref{lem:step2_general}. After that, the case of the star and the path are covered together; the proof in those cases uses the eating algorithm, Corollary~\ref{cor:eating algorithm}, to find adequate sequences of increasing subsets on which to build increasing sequences of \mstr s.

\begin{proof}[Proof of Proposition~\ref{prop:MAIN} (Case of the clique)]
    Using Lemma~\ref{lem:step2_general}, since $\mstr(\rK_n)$ is a subgraph of $H_n=K_n$, it follows that
    \begin{align*}
        \cost(\rK_n,H_n)\leq1+\wdiam\big(\mstr(\rK_n)\big)\,.
    \end{align*}
    By Theorem~\ref{thm:bound on diam(mst)} we have $\wdiam(\mstr(\rK_n)) \convp 0$, and the result follows.
\end{proof}

\subsection{MST sequences for the star and the path}

In this section, we assume that $H_n$ is either a star or a path. If $H_n$ is a star, then by relabeling we may assume $H_n$ has center $n$, so has edge set $\{e_1,\dots,e_{n-1}\}$ with $e_i=in$; call this star $\Star_n$. If $H_n$ is a path, then by relabeling we may assume $H_n$ is the path $\Path_n=12\dots n$, so has edge set $\{e_i,\dots,e_{n-1}\}$ with $e_i=i(i+1)$. In either case, with this edge labeling, for any $1 \leq i < j \leq n-1$, the set $V(i,j)$ defined as the endpoints in $\{e_i,\dots,e_{j-1}\}$ is connected in $H_n$. Note that $V(i,j)=\{i,\dots,j-1\} \cup \{n\}$ when $H_n$ is a star and $V(i,j)=\{1,\dots,j\}$ when $H_n$ is a path, and in both cases $|V(i,j)|=j-i+1$. For the remainder of the section, it might be helpful to imagine that $H_n$ is the path, $12\dots n$.

Recall that $\rX=(X_e,e \in \rE(K_n))$ is a set of independent $\mathrm{Uniform}[0,1]$ random variables. For $\rW\in(0,1)$ and $2\leq\rL<n-1$, let 
\begin{align}
    \rI=\rI(\rW,\rL)=(n-\rL) \wedge \min\Big\{i:\forall i\leq j<i+\rL, X_{e_j}\leq \rW\Big\}\,. \label{eq:def I}
\end{align}
Note that $\rI$ is a function of $\rX$ and more precisely that
\begin{align*}
    \{\rI\leq k\}\in\sigma\Big(\big\{X_{e_i}\leq\rW\big\},1\leq i<k+\rL\Big)\,,
\end{align*}
where $\sigma(X)$ is the $\sigma$-algebra generated by $X$.

Next, let $\rU=\rU(\rI)=(U_i,0\leq i< n-\rL)$ be the sequence of sets defined as follows.
\begin{align}
    (U_0,\ldots,U_{n-\rL-1})=\Big(V(\rI,\rI+\rL),\ldots,V(\rI,n),V(\rI-1,n),\ldots,V(1,n)\Big)\,. \label{eq:definition of U}
\end{align}
In words, $U_0$ is the set of vertices that belong to the edges $e_\rI,\ldots,e_{\rI+\rL-1}$ (that is $V(\rI,\rI+\rL)$); then we sequentially build $U_1,\ldots,U_{n-\rL-1}$ by first adding the vertices belonging to $e_{\rI+\rL},\ldots,e_{n-1}$, then adding the vertices belonging to $e_{\rI-1},\ldots,e_1$; see Figure~\ref{fig:I and U} for a representation of $\rI$ and $\rU$.

\begin{figure}[htb]
    \centering
    \begin{tikzpicture}
        \node[anchor=south west,block,draw,fill=lightblock!30!white,inner sep=0.2cm,rounded corners] at (-3,1){$\rW=0.2\hspace{0.5cm}\rL=3$};
        
        \draw[line width=0.05cm, theme] (1,0) -- (9,0);
        \foreach \n in {1,...,9}{
            \node[draw,circle,theme,line width=0.05cm, fill=white, scale=1.3] at (\n,0){};
            \node[theme,scale=0.9] at (\n,0){$\mathbf{\n}$};
        }
        \node[lighttheme,scale=0.7,anchor=south] at (1.5,0){$\mathbf{0.6}$};
        \node[lighttheme,scale=0.7,anchor=south] at (2.5,0){$\mathbf{0.9}$};
        \node[lighttheme,scale=0.7,anchor=south] at (3.5,0){$\mathbf{0.5}$};
        \node[lighttheme,scale=0.7,anchor=south] at (4.5,0){$\mathbf{0.1}$};
        \node[lighttheme,scale=0.7,anchor=south] at (5.5,0){$\mathbf{0.2}$};
        \node[lighttheme,scale=0.7,anchor=south] at (6.5,0){$\mathbf{0.1}$};
        \node[lighttheme,scale=0.7,anchor=south] at (7.5,0){$\mathbf{0.1}$};
        \node[lighttheme,scale=0.7,anchor=south] at (8.5,0){$\mathbf{0.7}$};
        \node[anchor=east,scale=1.2] at (0.5,0){$(\textcolor{theme}{H_n},\textcolor{lighttheme}{w})=$};
        \node[anchor=east,scale=0.8, text width=2.5cm,align=flush right] at (12,0){the ordered line with random edge weights.};

        \begin{scope}[yshift=-2.5cm]
            \draw[line width=0.05cm, theme] (1,0) -- (9,0);
            \foreach \n in {1,...,9}{
                \node[draw,circle,theme,line width=0.05cm, fill=white, scale=1.3] at (\n,0){};
                \node[theme,scale=0.9] at (\n,0){$\mathbf{\n}$};
            }
            \node[lighttheme,scale=0.7,anchor=south] at (1.5,0){$\mathbf{0.6}$};
            \node[lighttheme,scale=0.7,anchor=south] at (2.5,0){$\mathbf{0.9}$};
            \node[lighttheme,scale=0.7,anchor=south] at (3.5,0){$\mathbf{0.5}$};
            \node[lighttheme,scale=0.7,anchor=south] at (4.5,0){$\mathbf{0.1}$};
            \node[lighttheme,scale=0.7,anchor=south] at (5.5,0){$\mathbf{0.2}$};
            \node[lighttheme,scale=0.7,anchor=south] at (6.5,0){$\mathbf{0.1}$};
            \node[lighttheme,scale=0.7,anchor=south] at (7.5,0){$\mathbf{0.1}$};
            \node[lighttheme,scale=0.7,anchor=south] at (8.5,0){$\mathbf{0.7}$};
            
            \fill[opacity=0.7,white] (0.5,0.5) rectangle (9.5,-0.5);

            \draw[line width=0.05cm, theme] (4,0) -- (7,0);
            \foreach \n in {4,...,7}{
                \node[draw,circle,theme,line width=0.05cm, fill=white, scale=1.3] (\n) at (\n,0){};
                \node[theme,scale=0.9] at (\n,0){$\mathbf{\n}$};
            }
            \node[lighttheme,scale=0.7,anchor=south] at (4.5,0){$\mathbf{0.1}$};
            \node[lighttheme,scale=0.7,anchor=south] at (5.5,0){$\mathbf{0.2}$};
            \node[lighttheme,scale=0.7,anchor=south] at (6.5,0){$\mathbf{0.1}$};

            \node[block,anchor=east,draw,fill=lightblock!30!white,inner sep=0.15cm,rounded corners](I) at (3.5,-0.8){$\rI=4$};
            \draw[block,line width=0.03cm,->] (I.east) to[out=0,in=-90] (4.south);

            \node[anchor=east,scale=0.8,text width=2.5cm,align=flush right] at (12,0){$\rI$ is the first index followed by $\rL=3$ edges of weight less than $\rW=0.2$.};
        \end{scope}

        \begin{scope}[yshift=-5cm]
            \draw[block,fill=lightblock!30!white,rounded corners] (0,0.5) -- (9.5,0.5) -- (9.5,-5.5) -- (0,-5.5) -- (0,-3) -- (-1,-3) -- (-1,-2) -- (0,-2) -- cycle;
            \node[anchor=east,scale=1.2,block] at (0.2,-2.5){$\rU=$};
            \begin{scope}[yshift=-0cm]
                \draw[line width=0.05cm, theme!30!white] (1,0) -- (9,0);
                \foreach \n in {1,...,9}{
                    \node[draw,circle,theme!30!white,line width=0.05cm, fill=white, scale=1.3] at (\n,0){};
                    \node[theme!30!white,scale=0.9] at (\n,0){$\mathbf{\n}$};
                }
                \node[lighttheme!30!white,scale=0.7,anchor=south] at (1.5,0){$\mathbf{0.6}$};
                \node[lighttheme!30!white,scale=0.7,anchor=south] at (2.5,0){$\mathbf{0.9}$};
                \node[lighttheme!30!white,scale=0.7,anchor=south] at (3.5,0){$\mathbf{0.5}$};
                \node[lighttheme!30!white,scale=0.7,anchor=south] at (4.5,0){$\mathbf{0.1}$};
                \node[lighttheme!30!white,scale=0.7,anchor=south] at (5.5,0){$\mathbf{0.2}$};
                \node[lighttheme!30!white,scale=0.7,anchor=south] at (6.5,0){$\mathbf{0.1}$};
                \node[lighttheme!30!white,scale=0.7,anchor=south] at (7.5,0){$\mathbf{0.1}$};
                \node[lighttheme!30!white,scale=0.7,anchor=south] at (8.5,0){$\mathbf{0.7}$};
                
                
                \draw[line width=0.05cm, block] (4,0) -- (7,0);
                \foreach \n in {4,...,7}{
                    \node[draw,circle,block,line width=0.05cm, fill=lightblock, scale=1.3] at (\n,0){};
                    \node[block,scale=0.9] at (\n,0){$\mathbf{\n}$};
                }
                \node[block,anchor=north] at (3.5,0){$U_0$};
            \end{scope}
            \begin{scope}[yshift=-1cm]
                \draw[line width=0.05cm, theme!30!white] (1,0) -- (9,0);
                \foreach \n in {1,...,9}{
                    \node[draw,circle,theme!30!white,line width=0.05cm, fill=white, scale=1.3] at (\n,0){};
                    \node[theme!30!white,scale=0.9] at (\n,0){$\mathbf{\n}$};
                }
                \node[lighttheme!30!white,scale=0.7,anchor=south] at (1.5,0){$\mathbf{0.6}$};
                \node[lighttheme!30!white,scale=0.7,anchor=south] at (2.5,0){$\mathbf{0.9}$};
                \node[lighttheme!30!white,scale=0.7,anchor=south] at (3.5,0){$\mathbf{0.5}$};
                \node[lighttheme!30!white,scale=0.7,anchor=south] at (4.5,0){$\mathbf{0.1}$};
                \node[lighttheme!30!white,scale=0.7,anchor=south] at (5.5,0){$\mathbf{0.2}$};
                \node[lighttheme!30!white,scale=0.7,anchor=south] at (6.5,0){$\mathbf{0.1}$};
                \node[lighttheme!30!white,scale=0.7,anchor=south] at (7.5,0){$\mathbf{0.1}$};
                \node[lighttheme!30!white,scale=0.7,anchor=south] at (8.5,0){$\mathbf{0.7}$};
                
                
                \draw[line width=0.05cm, block] (4,0) -- (8,0);
                \foreach \n in {4,...,8}{
                    \node[draw,circle,block,line width=0.05cm, fill=lightblock, scale=1.3] at (\n,0){};
                    \node[block,scale=0.9] at (\n,0){$\mathbf{\n}$};
                }
                \node[block,anchor=north] at (3.5,0){$U_1$};
            \end{scope}
            \begin{scope}[yshift=-2cm]
                \draw[line width=0.05cm, theme!30!white] (1,0) -- (9,0);
                \foreach \n in {1,...,9}{
                    \node[draw,circle,theme!30!white,line width=0.05cm, fill=white, scale=1.3] at (\n,0){};
                    \node[theme!30!white,scale=0.9] at (\n,0){$\mathbf{\n}$};
                }
                \node[lighttheme!30!white,scale=0.7,anchor=south] at (1.5,0){$\mathbf{0.6}$};
                \node[lighttheme!30!white,scale=0.7,anchor=south] at (2.5,0){$\mathbf{0.9}$};
                \node[lighttheme!30!white,scale=0.7,anchor=south] at (3.5,0){$\mathbf{0.5}$};
                \node[lighttheme!30!white,scale=0.7,anchor=south] at (4.5,0){$\mathbf{0.1}$};
                \node[lighttheme!30!white,scale=0.7,anchor=south] at (5.5,0){$\mathbf{0.2}$};
                \node[lighttheme!30!white,scale=0.7,anchor=south] at (6.5,0){$\mathbf{0.1}$};
                \node[lighttheme!30!white,scale=0.7,anchor=south] at (7.5,0){$\mathbf{0.1}$};
                \node[lighttheme!30!white,scale=0.7,anchor=south] at (8.5,0){$\mathbf{0.7}$};
                
                
                \draw[line width=0.05cm, block] (4,0) -- (9,0);
                \foreach \n in {4,...,9}{
                    \node[draw,circle,block,line width=0.05cm, fill=lightblock, scale=1.3] at (\n,0){};
                    \node[block,scale=0.9] at (\n,0){$\mathbf{\n}$};
                }
                \node[block,anchor=north] at (3.5,0){$U_2$};
            \end{scope}
            \begin{scope}[yshift=-3cm]
                \draw[line width=0.05cm, theme!30!white] (1,0) -- (9,0);
                \foreach \n in {1,...,9}{
                    \node[draw,circle,theme!30!white,line width=0.05cm, fill=white, scale=1.3] at (\n,0){};
                    \node[theme!30!white,scale=0.9] at (\n,0){$\mathbf{\n}$};
                }
                \node[lighttheme!30!white,scale=0.7,anchor=south] at (1.5,0){$\mathbf{0.6}$};
                \node[lighttheme!30!white,scale=0.7,anchor=south] at (2.5,0){$\mathbf{0.9}$};
                \node[lighttheme!30!white,scale=0.7,anchor=south] at (3.5,0){$\mathbf{0.5}$};
                \node[lighttheme!30!white,scale=0.7,anchor=south] at (4.5,0){$\mathbf{0.1}$};
                \node[lighttheme!30!white,scale=0.7,anchor=south] at (5.5,0){$\mathbf{0.2}$};
                \node[lighttheme!30!white,scale=0.7,anchor=south] at (6.5,0){$\mathbf{0.1}$};
                \node[lighttheme!30!white,scale=0.7,anchor=south] at (7.5,0){$\mathbf{0.1}$};
                \node[lighttheme!30!white,scale=0.7,anchor=south] at (8.5,0){$\mathbf{0.7}$};
                
                
                \draw[line width=0.05cm, block] (3,0) -- (9,0);
                \foreach \n in {3,...,9}{
                    \node[draw,circle,block,line width=0.05cm, fill=lightblock, scale=1.3] at (\n,0){};
                    \node[block,scale=0.9] at (\n,0){$\mathbf{\n}$};
                }
                \node[block,anchor=north] at (2.5,0){$U_3$};
            \end{scope}
            \begin{scope}[yshift=-4cm]
                \draw[line width=0.05cm, theme!30!white] (1,0) -- (9,0);
                \foreach \n in {1,...,9}{
                    \node[draw,circle,theme!30!white,line width=0.05cm, fill=white, scale=1.3] at (\n,0){};
                    \node[theme!30!white,scale=0.9] at (\n,0){$\mathbf{\n}$};
                }
                \node[lighttheme!30!white,scale=0.7,anchor=south] at (1.5,0){$\mathbf{0.6}$};
                \node[lighttheme!30!white,scale=0.7,anchor=south] at (2.5,0){$\mathbf{0.9}$};
                \node[lighttheme!30!white,scale=0.7,anchor=south] at (3.5,0){$\mathbf{0.5}$};
                \node[lighttheme!30!white,scale=0.7,anchor=south] at (4.5,0){$\mathbf{0.1}$};
                \node[lighttheme!30!white,scale=0.7,anchor=south] at (5.5,0){$\mathbf{0.2}$};
                \node[lighttheme!30!white,scale=0.7,anchor=south] at (6.5,0){$\mathbf{0.1}$};
                \node[lighttheme!30!white,scale=0.7,anchor=south] at (7.5,0){$\mathbf{0.1}$};
                \node[lighttheme!30!white,scale=0.7,anchor=south] at (8.5,0){$\mathbf{0.7}$};
                
                
                \draw[line width=0.05cm, block] (2,0) -- (9,0);
                \foreach \n in {2,...,9}{
                    \node[draw,circle,block,line width=0.05cm, fill=lightblock, scale=1.3] at (\n,0){};
                    \node[block,scale=0.9] at (\n,0){$\mathbf{\n}$};
                }
                \node[block,anchor=north] at (1.5,0){$U_4$};
            \end{scope}
            \begin{scope}[yshift=-5cm]
                \draw[line width=0.05cm, theme!30!white] (1,0) -- (9,0);
                \foreach \n in {1,...,9}{
                    \node[draw,circle,theme!30!white,line width=0.05cm, fill=white, scale=1.3] at (\n,0){};
                    \node[theme!30!white,scale=0.9] at (\n,0){$\mathbf{\n}$};
                }
                \node[lighttheme!30!white,scale=0.7,anchor=south] at (1.5,0){$\mathbf{0.6}$};
                \node[lighttheme!30!white,scale=0.7,anchor=south] at (2.5,0){$\mathbf{0.9}$};
                \node[lighttheme!30!white,scale=0.7,anchor=south] at (3.5,0){$\mathbf{0.5}$};
                \node[lighttheme!30!white,scale=0.7,anchor=south] at (4.5,0){$\mathbf{0.1}$};
                \node[lighttheme!30!white,scale=0.7,anchor=south] at (5.5,0){$\mathbf{0.2}$};
                \node[lighttheme!30!white,scale=0.7,anchor=south] at (6.5,0){$\mathbf{0.1}$};
                \node[lighttheme!30!white,scale=0.7,anchor=south] at (7.5,0){$\mathbf{0.1}$};
                \node[lighttheme!30!white,scale=0.7,anchor=south] at (8.5,0){$\mathbf{0.7}$};
                
                
                \draw[line width=0.05cm, block] (1,0) -- (9,0);
                \foreach \n in {1,...,9}{
                    \node[draw,circle,block,line width=0.05cm, fill=lightblock, scale=1.3] at (\n,0){};
                    \node[block,scale=0.9] at (\n,0){$\mathbf{\n}$};
                }
                \node[block,anchor=north] at (0.5,0){$U_5$};
            \end{scope}

            \node[anchor=north east,scale=0.8,text width=2.7cm,align=flush right] at (12,0.5){The sets in $\rU$ are built using $\rL=3$ and $\rI=4$ to set $U_0=\{4,5,6,7\}$ before expanding on both sides (right then left).};
        \end{scope}

        \node[scale=2.5] (top) at (1,0){};
        \node[scale=2.5] (middle) at (1,-2.5){};
        \node[scale=2.5] (bottom) at (1,-5){};
        \draw[->,line width=0.05cm,alert] (top) to[out=-135,in=135] (middle);
        \draw[->,line width=0.05cm,alert] (middle) to[out=-135,in=135] (bottom);
    \end{tikzpicture}
    \caption{An example of $\rI$ and $\rU$ for an instance of the weighted ordered line $(H_n,w)$, with $\rW=0.2$ and $\rL=3$. First, $\rI$ is set to be the first sequence of $\rL=3$ consecutive edges with weights less than $\rW=0.2$. In this example, $\rI=4$. Then, given $\rI$, set $U_0=V(\rI,\rI+\rL)=\{4,5,6,7\}$ and expand first to the right and then to left to obtain $U_1,\ldots,U_{n-\rL-1}$. In other words, in order to obtain $U_1$, $U_2$, $U_3$, $U_4$, and $U_5$, we sequentially add $8$, $9$, $3$, $2$, and $1$ to $U_0$.}
    \label{fig:I and U}
\end{figure}

We now use the sequence $\rU$ to bound the cost of $(\rK_n,H_n)$ when $H_n$ is a star or a path. The following lemma gives a first bound on the cost using $\rU$.

\begin{lemma}\label{lem:cost to U}
    Let $H_n$ be the star $\Star_n$ or path $\Path_n$. Then, conditionally given that $\rI(\rW,\rL)<n-\rL$, we have
    \begin{align*}
        \cost(\rK_n,H_n)\leq\max\bigg\{\rW\rL,1+\max\Big\{\wdiam\big(\mstr(\rK_n[U_{i}])\big):0\leq i<n-\rL\Big\}\bigg\}\,.
    \end{align*}
\end{lemma}

\begin{proof}
    This result almost directly follows from Corollary~\ref{cor:eating algorithm}. Indeed, let $H'_n=\Phi(H_n,U_{0})$. Then the sets $U_{0},\ldots,U_{n-\rL-1}$ satisfy the condition of Corollary~\ref{cor:eating algorithm} with $H = H'_n$, implying that
    \begin{align*}
        \cost(\rK_n,H'_n)\leq1+\max\Big\{\wdiam\big(\mstr(\rK_n[U_{i}])\big):0\leq i<n-\rL\Big\}\,.
    \end{align*}
    But now, by concatenating any minimal weight \mstr\ sequence for $(\rK_n[U_0],H_n[U_0])$ and any minimal weight \mstr\ sequence for $(\rK_n,H'_n)$, it follows that
    \begin{align*}
        \cost(\rK_n,H_n)\leq\max\Big\{\cost\big(\rK_n[U_0],H_n[U_0]\big),\cost\big(\rK_n,H'_n\big)\Big\}\,.
    \end{align*}
    In order to complete the proof of the lemma, note that, conditionally given $\rI<n-\rL$,
    \begin{align*}
        \mathop{w}(H_n[U_0])=\sum_{e\in\rE(H_n[U_0])}X_e\leq\rW\rL\,.
    \end{align*}
    Taking $\rS=(U_0)$, this yields
    \begin{align*}
        \cost\big(\rK_n[U_0],H_n[U_0]\big)\leq\wt\big(\rK_n[U_0],H_n[U_0],\rS\big)=w\big(H_n[U_0]\big)\leq\rW\rL\,.
    \end{align*}
    This proves the desired upper bound and concludes the proof of the lemma.
\end{proof}

The next two results, combined with Lemma~\ref{lem:cost to U}, will allow us to give the full proof of Proposition~\ref{prop:MAIN} when $H_n$ is either a star or a path.

\begin{prop}\label{prop:good sets star path}
    For any $\epsilon>0$, for $\rW=\frac{1}{\log n}$ and $\rL=\lfloor\log\log n\rfloor$, as $n \rightarrow \infty$ we have
    \begin{align*}
        \mathbb{P}\Big(\exists U\in\rU(\rI(\rW,\rL)):\wdiam\big(\mstr(\rK_n[U])\big)>\epsilon \Big)\longrightarrow0\,.
    \end{align*}
\end{prop}
\begin{lemma}\label{lem:bound I}
    Let $\rW=\frac{1}{\log n}$ and $\rL=\lfloor\log\log n\rfloor$. Then, for any $a>0$, as $n\rightarrow\infty$ we have
    \begin{align*}
        \mathbb{P}\big(\rI(\rW,\rL)\geq n^a\big)\longrightarrow0\,.
    \end{align*}
\end{lemma}

Lemma~\ref{lem:bound I} is straightforward and we prove it immediately. On the other hand, Proposition~\ref{prop:good sets star path} is quite technical and we dedicate Section~\ref{sec:good sets star path} below to proving it.

\begin{proof}[Proof of Lemma~\ref{lem:bound I}]
    For any integer $k\geq1$, by the definition of $\rI$,
    \begin{align*}
        \mathbb{P}\big(\rI\geq k\rL+1\big)&=\mathbb{P}\Big(\forall i<k\rL+1,\exists j\in\{i,\ldots,i+\rL-1\}:X_{e_j}>\rW\Big)\\
        &\leq\mathbb{P}\Big(\forall i\in\big\{1,1+\rL,\ldots,1+(k-1)\rL\big\},\exists j\in\{i,\ldots,i+\rL-1\}:X_{e_j}>\rW)\Big)\,.
    \end{align*}
    But then, by independence of the weights of $\rX$, we have
    \begin{align*}
        \mathbb{P}\big(\rI\geq k\rL+1\big)&=\prod_{i=0}^{k-1}\mathbb{P}\Big(\exists j\in\{1+i\rL,\ldots,1+(i+1)\rL-1\}:X_{e_j}>\rW\Big)=\prod_{i=0}^{k-1}\Big(1-\rW^\rL\Big)\leq e^{-k\rW^\rL}\,,
    \end{align*}
    where the last inequality follows from the convexity of the exponential.
    Applying this result with $k=\lfloor\frac{n^a-1}{\rL}\rfloor$, we obtain
    \begin{align*}
        \mathbb{P}\big(\rI\geq n^a\big)\leq\mathbb{P}\big(\rI\geq k\rL+1\big)\leq \exp\left(- \left\lfloor\frac{n^a-1}{\rL}\right\rfloor \cdot \rW^{\rL}\right) \,,
    \end{align*}
    and the final expression tends to 0 as $n \rightarrow \infty$.
\end{proof}

\begin{proof}[Proof of Proposition~\ref{prop:MAIN} (Case of the star and the path)]
     Let $\rW=\frac{1}{\log n}$ and  $\rL=\lfloor\log\log n\rfloor$. Fixing $\eps > 0$, we have
    \begin{align*}
        \mathbb{P}\Big(\cost(\rK_n,H_n)>1+\epsilon\Big)\leq\mathbb{P}\Big(\cost(\rK_n,H_n)>1+\epsilon~\Big|~\rI<n-\rL\Big)+\mathbb{P}\Big(\rI=n-\rL\Big)\,.
    \end{align*}
    Applying Lemma~\ref{lem:bound I} with any $a<1$, for large enough $n$ we have
    \begin{align*}
        \mathbb{P}\Big(\rI=n-\rL\Big)\leq\mathbb{P}\Big(\rI\geq n^a\Big)\longrightarrow0\,.
    \end{align*}
    Hence, we have
    \begin{align*}
        \mathbb{P}\Big(\cost(\rK_n,H_n)>1+\epsilon\Big)=\mathbb{P}\Big(\cost(\rK_n,H_n)>1+\epsilon~\Big|~\rI<n-\rL\Big)+o(1)\,.
    \end{align*}
    Since $\rW\rL\to0$, combining the previous bound with Lemma~\ref{lem:cost to U} leads to
    \begin{align*}
        &\mathbb{P}\Big(\cost(\rK_n,H_n)>1+\epsilon\Big)\\
        &\hspace{0.5cm}\leq\mathbb{P}\left(\max\bigg\{\rW\rL,1+\max\Big\{\wdiam\big(\mstr(\rK_n[U_{i}])\big)\Big\}\bigg\}>1+\epsilon~\bigg|~\rI<n-\rL\right)+o(1)\\
        &\hspace{0.5cm}=\mathbb{P}\left(\max\Big\{\wdiam\big(\mstr(\rK_n[U])\big):U\in\rU\Big\}>\epsilon~\bigg|~\rI<n-\rL\right)+o(1)\,.
    \end{align*}
    The upper bound now follows from Proposition~\ref{prop:good sets star path}, once again since $\mathbb{P}(\rI<n-\rL)\rightarrow0$.
\end{proof}

\subsection{Proof of Proposition~\ref{prop:good sets star path}}\label{sec:good sets star path}

In this section, we prove Proposition~\ref{prop:good sets star path}, which concludes the proof of Proposition~\ref{prop:MAIN}. Before doing so, we state a proposition which is an important input to the proof.

\begin{prop}\label{prop:reduced MST}
    Let $\rG=(G,w)$ be a weighted graph. Let $T$ be a subtree (not necessarily spanning) of $\rG$ and let $\rG^*=(G,w^*)$ be a weighted graph such that $\mathop{w^*}(e)\leq \mathop{w}(e)$ for $e\in\rE(T)$ and $\mathop{w^*}(e)=\mathop{w}(e)$ otherwise. Then
    \begin{align*}
        \wdiam\big(\mstr(\rG^*)\big)\leq \mathop{w^*}(T)+|\rV(T)|\times\wdiam\big(\mstr(\rG)\big)\,.
    \end{align*}
    Moreover, if $T$ is a subtree of $\mstr(\rG^*)$, then
    \begin{align*}
        \wdiam\big(\mstr(\rG^*)\big)\leq \mathop{w^*}(T)+2\times\wdiam\big(\mstr(\rG)\big)\,.
    \end{align*}
\end{prop}

\begin{proof}
Let us try to understand the relation between $\mstr(\rG)$ and $\mstr(\rG^*)$. First note that
\begin{align}
    \rE(\mstr(\rG^*))\subset\rE\big(\mstr(\rG)\big)\cup\rE(T)\,. \label{eq:reduced weight mst relation}
\end{align}
Indeed, any edge $e\notin\rE(T)$ has the same weight with respect to $w$ and $w^*$. Then, for any $e\in\rE(\mstr(\rG^*))\setminus\rE(T)$, no cycle has $e$ as the heaviest edge with respect to $w^*$, which implies that no cycle has $e$ as the heaviest edge with respect to $w$, and thus $e\in\rE(\mstr(\rG))$.

Consider now a path $P$ contained in $\mstr(\rG^*)$. Using (\ref{eq:reduced weight mst relation}), we have
\begin{align*}
    \rE(P)\subseteq \rE\big(\mstr(\rG)\big)\cup\rE(T) \,,
\end{align*}
so we may uniquely decompose $P$ into pairwise edge-disjoint paths $P_0,\dots,P_{2k}$, where $k \geq 1$, and $P_i$ is a subpath of $T$ for $i$ odd and of $\mstr(\rG)$ for $i$ even (it is possible that either or both of $P_0,P_{2k}$ consists of a single vertex). Since $P_1,P_3,\dots,P_{2k-1}$ are disjoint subpaths of $T$, it follows that $k \leq |\rE(T)|$ and that $\sum_{i \text{ odd}} \mathop{w}(P_i) \leq \mathop{w}(T)$. Moreover, each of the paths $P_0,P_2,\dots,P_{2k}$ have weight at most $\wdiam\big(\mstr(\rG)\big)$, so
\begin{align}
    \sum_{i \text{ even}} \mathop{w}(P_i) &\leq (k+1) \times \wdiam\big(\mstr(\rG)\big) \label{eq:sum of path weights}\\
    &\leq \big( |\rE(T)|+1 \big) \times \wdiam\big(\mstr(\rG)\big) \notag\\
    &= |\rV(T)| \times \wdiam\big(\mstr(\rG)\big) \,.\notag 
\end{align}
The first bound of the proposition follows since 
\begin{align*}
    \mathop{w}(P) = \sum_{i \text{ even}} \mathop{w}(P_i) + \sum_{i \text{ odd}} \mathop{w}(P_i) \,.
\end{align*}
To establish the second bound, note that if $T$ is a subtree of $\mstr(\rG^*)$ then in the above decomposition of $P$ we must have $k=1$; a path in $\mstr(\rG^*)$ may enter $T$ and then leave it, after which it can never reenter $T$. In this case the first summation of (\ref{eq:sum of  path weights}) becomes
\begin{align*}
    \sum_{i \text{ even}} \mathop{w}(P_i) \leq 2 \times \wdiam\big(\mstr(\rG)\big) \,,
\end{align*}
so we obtain 
\begin{align*}
    \mathop{w}(P) = \sum_{i \text{ even}} \mathop{w}(P_i) + \sum_{i \text{ odd}} \mathop{w}(P_i) \leq \mathop{w}(T) + 2 \times \wdiam(\mstr(\rG)) \,,
\end{align*}
as required.
\end{proof}

For the remainder of this section we assume $\rW = \Wvalue$ and $\rL = \Lvalue$ and write $\rI=\rI(\rW,\rL)$. Consider the partition $\rU = \rU_r^- \cup \rU_r^+ \cup \rU_{\ell}$ where $\rU_r^-=\rU_r^-(\rI)=(U_{i},0\leq i\leq\min(\rL^{20},n-\rI-\rL))$, $\rU_r^+=\rU_r^+(\rI)=(U_{i},\min(\rL^{20},n-\rI-\rL)<i\leq n-\rI-\rL)$, and $\rU_\ell=\rU_\ell(\rI)=(U_{i},n-\rI-\rL<i\leq n-\rL-1)$. Then, in the case where $\rI<n-\rL-\rL^{20}$, $\rU_r^-$ corresponds to adding the first $\rL^{20}$ vertices on the right of $U_{0}$, $\rU_r^+$ corresponds to adding all remaining vertices on the right, and $\rU_\ell$ corresponds to adding the vertices on the left of $U_{0}$. We aim to prove tail bounds similar to that of Proposition~\ref{prop:good sets star path} for each of the sets $\rU_r^-$, $\rU_r^+$, and $\rU_\ell$, and we start with an important lemma regarding the distribution of $\rG$ conditioned on the value of $\rI$.

\begin{lemma}\label{lem:I coupling}
    Fix $k<n-\rL$ and let $\rK^*_n=(K_n,\rX^*)$ have the law of $\rK_n$ conditioned on the event that $\rI(\rW,\rL)=k$.
    Then for any $e \in \{e_i,k\leq i<k+\rL\}$, $X^*_e$ is a $\mathrm{Uniform}[0,\rW]$; for any $e\notin\{e_{i}:1\leq i<k+\rL\}$, $X^*_{e_i}$ is a random $\mathrm{Uniform}[0,1]$, and the edge weights $\big( X_e^*, e\in \rE(K_n) \setminus\{e_i,1\leq i<k\}\big)$ are mutually independent and independent of $(X^*_e,e\in\{e_i,1\leq i<k\})$. It follows that there exists a coupling between $\rK^*_n=(K_n,\rX^*)$ and $\rK'_n=(K_n,\rX')$  where $\rX'$ is a set of independent $\mathrm{Uniform}[0,1]$, such that $X^*_e\leq X'_e$ if $e\in\{e_{i}:k\leq i<k+\rL\}$, and $X^*_e=X'_e$ if $e\in\rE(K_n)\setminus\{e_{i}:1\leq i<k+\rL\}$.
\end{lemma}

\begin{proof}
    Using the definition of $\rI$, we know that
    \begin{align*}
        \big\{\rI=k\big\}\in\sigma\Big(\big\{X_{e_i}\leq\rW:1\leq i<k+\rL\big\}\Big)\,,
    \end{align*}
    from which it directly follows that the distribution of $X_{e}$ is a Uniform$[0,1]$ for any $e\notin\{e_{n,i}:1\leq i<k+\rL\}$. Furthermore, for any $e\in\{e_i:k\leq i<k+\rL\}$, $X_e$ conditioned on $\rI=k$ is the same as $X_e$ conditioned on $X_e\leq\rW$. Since $X_e$ is uniformly distributed, it follows that $X_e$ conditioned on $\rI=k$ is a Uniform$[0,\rW]$. Finally, note that
    \begin{align*}
        \big\{\rI=k\big\}=\big\{X_{e_i}\leq\rW:k\leq i<k+\rL\big\}\cap\bigcap_{j=1}^{k-1}\Big\{\exists j\leq i<\min\{j+\rL,k\}:X_{e_i}>\rW\Big\}\,,
    \end{align*}
    from which we see that the edges of $\rE(K_n)\setminus\{e_i,1\leq i< k\}$ are conditionally independent of $\{e_i,1\leq i< k\}$ given that $\rI=k$. It follows that all the edges in $\rE(\rK_n)\setminus\{e_i,1\leq i<k\}$ have independent weights in $\rK^*_n$. The existence of the coupling asserted in the lemma is then an immediate consequence.
\end{proof}

We now split the proof of Proposition~\ref{prop:good sets star path} into proving analogous statements for the three different sets $\rU_r^-$, $\rU_r^+$, and $\rU_\ell$.

\paragraph{First right set $\rU_r^-$.}

\begin{lemma}\label{lem:close right}
    For any $\epsilon>0$, we have
    \begin{align*}
        \mathbb{P}\Big(\exists U \in\rU_r^-:\wdiam\big(\mstr(\rK_n[U])\big)>\epsilon\Big)\longrightarrow0\,.
    \end{align*}
\end{lemma}

\begin{proof}
    Fix $0 < a < 1$ and assume $n$ is large enough so that $n^a < n-\rL-\rL^{20}$. Then, by Lemma~\ref{lem:bound I}, we have
    \begin{align*}
        &\hspace{-0.5cm}\mathbb{P}\Big(\exists U \in\rU_r^-:\wdiam\big(\mstr(\rK_n[U])\big)>\epsilon\Big) \\ 
        \leq \ &
        \mathbb{P}\Big(\exists U \in\rU_r^-:\wdiam\big(\mstr(\rK_n[U])\big)>\epsilon ~\Big|~ \rI < n-\rL-\rL^{20} \Big) + \mathbb{P}\big( \rI \geq n-\rL-\rL^{20} \big) \\
        = \ & 
        \mathbb{P}\Big(\exists U \in\rU_r^-:\wdiam\big(\mstr(\rK_n[U])\big)>\epsilon ~\Big|~ \rI < n-\rL-\rL^{20} \Big) + o(1) \,.
    \end{align*}
    
    Next fix $k<n-\rL-\rL^{20}$ and condition on the event $\rI=k$. Under this conditioning, $\rU_r^-=\rU_r^-(\rI)=\rU_r^-(k)$ is a deterministic sequence of sets. Further recall from \eqref{eq:definition of U} that $U_0 = V(k,k+\rL)$ consists of the endpoints of the edges $e_k,\dots,e_{k+\rL-1}$, so equals $\{k,\dots,k+\rL\}$ if $H_n$ is the path $\Path_n$ and equals $\{k,\dots,k+\rL-1,n\}$ if $H_n$ is the star $\Star_n$. Let $T=H_n[U_0]$. Since $\rI = k < n-\rL$, all edges in $T$ have weight less than $\rW$. Now, suppose that all other edges of $\rK_n[U_{\rL^{20}}]$ have weight larger than $\rW$. In this case, $T$ is a subtree of $\mstr(\rK_n[U_{\rL^{20}}])$, from which it follows that $T$ is a subtree of $\mstr(\rK_n[U_i])$ for any $0 \leq i \leq \rL^{20}$ (since $U_i \subset U_{\rL^{20}}$ for such $U_i$). Now, using that $\{\rI=k\}\in\sigma(\{X_{e_i}:1\leq i<k+\rL\})$, we have
    \begin{align*}
        \mathbb{P}\Big(\forall e\in\rE\big(\rK_n[U_{\rL^{20}}]\big)\setminus\rE(T),X_e>\rW~\Big|~\rI=k\Big)=\big(1-\rW\big)^{\binom{\rL^{20}}{2}-\rL}\,.
    \end{align*}
    Since $\rW = \Wvalue$, we have $1-\rW \geq \exp(-2\rW)$ for $n$ large, so
    \begin{align*}
        \mathbb{P}\Big(\rE(T)\subset\rE\big(\mstr\big(\rK_n[U_{\rL^{20}}]\big)\big)~\Big|~\rI=k\Big)
        &\geq
        \big(1-\rW\big)^{\binom{\rL^{20}}{2}-\rL}\\
        &\geq
        \exp\left(-2\rW \left(\binom{\rL^{20}}{2} - \rL \right)\right)\\
        &\geq
        \exp \left( -\rW \rL^{40} \right)\\
        &\geq 
        1 - \frac{(\log \log n)^{40}}{\log n} \,,
    \end{align*}
    the last inequality holding since $\rW=\frac{1}{\log n}$, $\rL=\lfloor\log\log n\rfloor$, and $e^{-x} \geq 1-x$ for $x \geq 0$. Hence,
    \begin{align}
        &\mathbb{P}\Big(\exists U \in\rU_r^-:\wdiam\big(\mstr(\rK_n[U])\big)>\epsilon~\Big|~\rI=k\Big)\label{eq:bound Ur-}\\
        &\hspace{0.5cm}\leq\mathbb{P}\Big(\exists U \in\rU_r^-:\wdiam\big(\mstr(\rK_n[U])\big)>\epsilon,\rE(T)\subset\rE\big(\mstr\big(\rK_n[U_{\rL^{20}}]\big)\big)~\Big|~\rI=k\Big)+\frac{(\log \log n)^{40}}{\log n} \,.\notag
    \end{align}
    
    Let $(\rK_n^*,\rK'_n)$ be as in Lemma~\ref{lem:I coupling}. By the definition of $\rK^*_n$ and \eqref{eq:bound Ur-}, we have that
    \begin{align*}
        &\mathbb{P}\Big(\exists U \in\rU_r^-:\wdiam\big(\mstr(\rK_n[U])\big)>\epsilon~\Big|~\rI=k\Big)\\
        &\hspace{0.5cm}\leq\mathbb{P}\Big(\exists U \in\rU_r^-(k):\wdiam\big(\mstr(\rK^*_n[U])\big)>\epsilon,\rE(T)\subset\rE\big(\mstr\big(\rK^*_n[U_{\rL^{20}}]\big)\big)\Big)+\frac{(\log \log n)^{40}}{\log n} \,.
    \end{align*}
    Now, note that if $\rE(T)\subset\rE(\mstr(\rK^*_n[U_{\rL^{20}}]))$, then for any $U\in\rU_r^-(k)$, $\rE(T)\subset\rE(\mstr(\rK^*_n[U]))$, since $\mstr(\rK^*_n[U_{\rL^{20}}])[U]$ is a subgraph of $\mstr(\rK^*_n[U])$. Applying Proposition~\ref{prop:reduced MST} to $\mstr(\rK^*_n[U])$ and $\mstr(\rK'_n[U])$, it follows that
    \begin{align*}
        &\hspace{-0.5cm}\mathbb{P}\Big(\exists U \in\rU_r^-(k):\wdiam\big(\mstr(\rK^*_n[U])\big)>\epsilon,\rE(T)\subset\rE\big(\mstr\big(\rK^*_n[U_{\rL^{20}}]\big)\big)\Big)\\
        &\leq\mathbb{P}\Big(\exists U \in\rU_r^-(k):\mathop{w^*}(T)+2\times\wdiam\big(\mstr(\rK'_n[U])\big)>\epsilon,\rE(T)\subset\rE\big(\mstr\big(\rK^*_n[U_{\rL^{20}}]\big)\big)\Big)\\
        &\leq\mathbb{P}\Big(\exists U \in\rU_r^-(k):\mathop{w^*}(T)+2\times\wdiam\big(\mstr(\rK'_n[U])\big)>\epsilon\Big)\,.
    \end{align*}
    Using that $\mathop{w^*}(T)\leq\rW\rL$ and combining the two previous inequalities yields the bound
    \begin{align}
        &\hspace{-0.5cm}\mathbb{P}\Big(\exists U \in\rU_r^-:\wdiam\big(\mstr(\rK_n[U])\big)>\epsilon~\Big|~\rI=k\Big)\label{eq:bound first right}\\
        &\leq\mathbb{P}\Big(\exists U \in\rU_r^-(k):\wdiam\big(\mstr(\rK'_n[U])\big)>(\epsilon-\rW\rL)/2\Big)+\frac{(\log\log n)^{40}}{\log n}\,.\notag
    \end{align}
    We can now replace $\rK'_n$ by $\rK_n$ since they are identically distributed. Furthermore, recall that Theorem~\ref{thm:bound on diam(mst)} states that, for $n$ sufficiently large, we have
    \begin{align*}
    \mathbb{P}\left(\wdiam\big(\mstr(\rK_n)\big) \geq \frac{7 \log^4 n}{n^{1/10}}\right) \le \frac{4}{n^{\log n}} \,.
    \end{align*}
    Since $\rL \rightarrow \infty$ and $\rW\rL \rightarrow 0$ as $n \rightarrow \infty$, and since any set $U\in\rU_r^-$ has size $|U|\geq|U_0|=\rL+1$, we can choose $n$ large enough so that, for any set $U\in\rU_r^-$, we have $7\log^4|U|/|U|^{1/10}\leq(\epsilon-\rW\rL)/2$. It follows that
    \begin{align*}
        &\hspace{-0.5cm}\mathbb{P}\Big(\exists U\in\rU_r^-(k):\wdiam\big(\mstr(\rK'_n[U])\big)>(\epsilon-\rW\rL)/2\Big)\\
        &\leq\mathbb{P}\left(\exists U\in\rU_r^-(k):\wdiam\big(\mstr(\rK_n[U])\big)\geq\frac{7\log^4|U|}{|U|^{1/10}}\right)\\
        &\leq\sum_{U\in\rU_r^-(k)}\frac{4}{|U|^{\log|U|}}
    \end{align*}
    The final step of the proof is to use that $\rU_r^-(k)=(U_i,0\leq i\leq \rL^{20})$ where $|U_i|=|U_0|+i=\rL+i+1$, along with the fact that $4/n^{\log n}\leq 1/n^2$ for $n$ large enough, to obtain that
    \begin{align*}
        \mathbb{P}\Big(\exists U\in\rU_r^-(k):\wdiam\big(\mstr(\rK'_n[U])\big)>(\epsilon-\rW\rL)/2\Big)&\leq\sum_{k=\rL+1}^{\rL+\rL^{20}+1}\frac{1}{k^2}\leq\frac{1}{\rL}\leq\frac{2}{\log\log n}\,.
    \end{align*}
    Plugging this into \eqref{eq:bound first right}, it follows that
    \begin{align*}
        \mathbb{P}\Big(\exists U\in\rU_r^-:\wdiam\big(\mstr(\rK_n[U])\big)>\epsilon~\Big|~\rI=k\Big) \leq \frac{(\log \log n)^{40}}{\log n} + \frac{2}{\log \log n} \,.
    \end{align*}
    Finally, since the previous inequality holds for any $k<n-\rL-\rL^{20}$, we have 
    \begin{align*}
        &\mathbb{P}\Big(\exists U\in\rU_r^-:\wdiam\big(\mstr(\rK_n[U])\big)>\epsilon\Big)\\
        &\hspace{0.5cm}=\mathbb{P}\Big(\exists U\in\rU_r^-:\wdiam\big(\mstr(\rK_n[U])\big)>\epsilon ~\Big|~ \rI < n - \rL - \rL^{20} \Big) + o(1)\\
        &\hspace{0.5cm}\leq \frac{(\log \log n)^{40}}{\log n} + \frac{1}{(\log \log n)^{20}} + o(1)\longrightarrow0\,,
    \end{align*}
    which is the desired result.
\end{proof}

\paragraph{Second right set $\rU_r^+$.}

\begin{lemma}\label{lem:far right}
    For any $\epsilon>0$, we have
    \begin{align*}
        \mathbb{P}\Big(\exists U\in\rU_r^+:\wdiam\big(\mstr(\rK_n[U])\big)>\epsilon\Big)\longrightarrow0\,.
    \end{align*}
\end{lemma}

\begin{proof}
    Fix $0 < a < 1$ and assume $n$ is large enough so that $n^a < n-\rL-\rL^{20}$. Then, by Lemma~\ref{lem:bound I}, we have
    \begin{align*}
        &\mathbb{P}\Big(\exists U \in\rU_r^+:\wdiam\big(\mstr(\rK_n[U])\big)>\epsilon\Big) \\ 
        &\hspace{0.5cm}\leq
        \mathbb{P}\Big(\exists U \in\rU_r^+:\wdiam\big(\mstr(\rK_n[U])\big)>\epsilon ~\Big|~ \rI < n-\rL-\rL^{20} \Big) + \mathbb{P}\big( \rI \geq n-\rL-\rL^{20}\big) \\
        &\hspace{0.5cm}= 
        \mathbb{P}\Big(\exists U \in\rU_r^+:\wdiam\big(\mstr(\rK_n[U])\big)>\epsilon ~\Big|~ \rI < n-\rL-\rL^{20} \Big) + o(1) \,.
    \end{align*}
    Fix now $k<n-\rL-\rL^{20}$ and condition on the event $\rI=k$. Let $T=H_n[U_0]$ and let $(\rK_n',\rK_n^*)$ be given by the coupling in Lemma~\ref{lem:I coupling}. Then, by Proposition~\ref{prop:reduced MST},
    \begin{align*}
        &\mathbb{P}\Big(\exists U \in\rU_r^+:\wdiam\big(\mstr(\rK_n[U])\big)>\epsilon ~\Big|~ \rI = k \Big)\\
        &\hspace{0.5cm}\leq \mathbb{P}\Big(\exists U\in\rU_r^+(k):\wdiam\big(\mstr(\rK_n^*[U])\big)>\epsilon\Big)\\
        &\hspace{0.5cm}\leq \mathbb{P}\Big(\exists U\in\rU_r^+(k):\mathop{w^*}(T) + |\rV(T)| \times \wdiam\big(\mstr(\rK_n'[U])\big)>\epsilon\Big)\\
        &\hspace{0.5cm}\leq \mathbb{P}\Big(\exists U\in\rU_r^+(k):\wdiam\big(\mstr(\rK_n[U])\big)>(\epsilon - \rW\rL)/(\rL + 1) \Big)\,,
    \end{align*}
    where the last step follows from the fact that $\mathop{w^*}(T)\leq\rW\rL$ conditionally given that $\rI<n-\rL$, that $|\rV(T)|=\rL+1$, and that $\rK'_n$ is distributed as $\rK_n$.
    Since $x\mapsto\frac{\log^3 x}{x^{1/10}}$ is a decreasing function for large enough $x$, since any set $U\in\rU_r^+$ has size $|U|\geq|U_{\rL^{20}}|=\rL+\rL^{20}+1$, and since $\rL=\lfloor\log\log n\rfloor\rightarrow\infty$ and $\rW\rL=\lfloor\log\log n\rfloor/\log n\rightarrow0$, we can choose $n$ large enough so that, for any $U\in\rU_r^+$
    \begin{align*}
        \frac{7 \log^4 |U|}{|U|^{1/10}} 
        \leq
        \frac{7 \log^4 (\rL^{20})}{(\rL^{20})^{1/10}}
        =
        \frac{7\cdot 20^4\log^4\cdot (\rL)}{\rL^{2}} 
        \leq 
        \frac{\epsilon - \rW\rL}{\rL+1} \,.
    \end{align*}
    Then, recalling that $\rU_r^+(k)=(U_i,\rL^{20}<i\leq n-k-\rL)$ where $|U_i|=|U_0|+i=\rL+i+1$, Theorem~\ref{thm:bound on diam(mst)} gives us
    \begin{align*}
        &\mathbb{P}\Big(\exists U\in\rU_r^+(k):\wdiam\big(\mstr(\rK_n[U])\big)>(\epsilon - \rW\rL)/(\rL + 1) \Big)\\
        &\hspace{0.5cm}\leq \mathbb{P}\left(\exists U\in\rU_r^+(k):\wdiam\big(\mstr(\rK_n[U])\big)>\frac{7 \log^4 |U|}{|U|^{1/10}}  \right)\\
        &\hspace{0.5cm}\leq \sum_{U\in\rU_r^+(k)}\frac{4}{|U|^{\log|U|}}\\
        &\hspace{0.5cm}\leq\frac{1}{\rL+\rL^{20}}\,,
    \end{align*}
    where the last inequality uses that $x^{\log x}\geq 4x^2$ for $x$ large enough, along with the fact that $|U_i|=\rL+i+1$.
    Therefore, 
    \begin{align*}
        &\mathbb{P}\Big(\exists U\in\rU_r^+:\wdiam\big(\mstr(\rK_n[U])\big)>\epsilon\Big)\\
        &\hspace{0.5cm}\leq \mathbb{P}\Big(\exists U\in\rU_r^+(k):\wdiam\big(\mstr(\rK_n[U])\big)>(\epsilon - \rW\rL)/(\rL + 1) \Big) + o(1)\\
        &\hspace{0.5cm}\leq \frac{1}{\rL+\rL^{20}} + o(1)\longrightarrow0\,,
    \end{align*}
    concluding the proof of the lemma.
\end{proof}

\paragraph{Left set $\rU_\ell$.}

\begin{lemma}\label{lem:left}
    For any $\epsilon>0$, we have
    \begin{align*}
        \mathbb{P}\Big(\exists U\in\rU_\ell:\wdiam\big(\mstr(\rK_n[U])\big)>\epsilon\Big)\longrightarrow0\,.
    \end{align*}
\end{lemma}

\begin{proof}
    Fix $a<\frac{1}{4}$. Thanks to Lemma~\ref{lem:bound I}, we know that $\mathbb{P}(\rI\geq n^a)\to0$. Moreover, note that under this event, any set $U\in\rU_\ell$ has size $|U|\geq n-k\geq n-n^a$. Our strategy now is to prove that, due to the large size of these sets, conditioning on the event $\{\rI<n^a\}$ does not notably affect the structure of $\mstr(\rK_n[U])$.
    
    Let us try to understand how the edge weights $\{e_1,\dots,e_{n-1}\}$ behave given that $\rI<n^a$; call $\rK^a_n$ the random weighted graph corresponding to the distribution of $\rK_n$ conditionally given that $\rI<n^a$. Recall that $\{\rI<n^a\}\in\sigma(\{X_{e_i}:1\leq i<\lceil n^a\rceil+\rL\})$ and write $m=\lceil n^a\rceil+\rL-1$ (note that $e_1,\dots,e_m$ are the only edges affected when we condition on $\rI < n^a$). Let $\rA=\{i\leq m:X_{e_i}\leq\rW\}$ and let $\mathcal{A}$ be the collection of sets $A\subset[m]$ such that there exists $i<n^a$ with $\{i,\ldots,i+\rL-1\}\subset A$. Then, by definition, $\{\rA\in\mathcal{A}\}=\{\rI<n^a\}$. Now, for any $A \in \mathcal{A}$, conditionally given that $\rA=A$, the weights of $e_1,\ldots,e_m$ are independent of each other and are distributed as Uniform$[0,W]$ or Uniform$[W,1]$, according to whether or not the index $i$ of the edge $e_i$ lies in $A$. This means that for any $x_1,\ldots,x_m\in[0,1]$, and any $A \in \mathcal{A}$, we have
    \begin{align*}
        \mathbb{P}\Big(\forall i\in[m]:X_{e_i}\leq x_i~\Big|~\rA=A,\rI<n^a\Big)&=\mathbb{P}\Big(\forall i\in[m]:X_{e_i}\leq x_i~\Big|~\rA=A\Big)\\
        &=\left(\prod_{i\in A}\frac{\min\{x_i,\rW\}}{\rW}\right)\left(\prod_{i\in[m]\setminus A}\frac{\max\{x_i,\rW\}-\rW}{1-\rW}\right)\,.
    \end{align*}
    Now, using that $\frac{\max\{x_i,\rW\}-\rW}{1-\rW}\leq\frac{\min\{x_i,\rW\}}{\rW}$, it follows that
    \begin{align*}
        \mathbb{P}\Big(\forall i\in[m]:X_{e_i}\leq x_i~\Big|~\rA=A,\rI<n^a\Big)&\leq\mathbb{P}\Big(\forall i\in[m]:X'_{e_i}\leq x_i\Big)\,,
    \end{align*}
    where $(X'_{e_1},\ldots,X'_{e_m})$ are independent Uniform$[0,\rW]$. This implies that there exists a generic weighted graph $\rK'_n=(K_n,\rX')$ with independent weights, where $X'_e$ is a Uniform$[0,1]$ if $e\notin\{e_1,\ldots,e_m\}$ and a Uniform$[0,\rW]$ otherwise, and a coupling between $\rK'_n$ and $\rK^a_n$ such that $X'_e\leq X^a_e$ for any $e\in\rE(K_n)$. We now use this coupling to prove the lemma.

    Consider the event
    \begin{align*}
        E'=\Big\{\forall k<n^a,\forall i\in[m],e_i\notin\rE\big(\mstr(\rK'_n[V(k,n)])\big)\Big\}
    \end{align*}
    By using two union bounds, we have that
    \begin{align*}
        \mathbb{P}(E')\geq1-\sum_{k<n^a}\sum_{i\in[m]}\mathbb{P}\Big(e_i\in\rE\big(\mstr(\rK'_n[V(k,n)])\big)\Big)\,.
    \end{align*}
    For $k$ and $i$ as in the above sum, if there exists $j\in V(k,n)\setminus e_{i}$ such that the weight of $e_{i}$ is larger than the weight of the two other edges in the triangle $\Delta_{i,j}$ formed by $e_{i}$ and $j$, then $e_{i}$ is not in the \mstr\ of $\rK'_n[V(k,n)]$. This means that
    \begin{align*}
        &\mathbb{P}\Big(e_{i}\in\rE\big(\mstr(\rK'_n[V(k,n)])\big)~\Big|~X'_{e_i}\Big)\\
        &\hspace{0.5cm}\leq\mathbb{P}\Big(\forall j\in V(k,n)\setminus e_{i}, \max(X_e' : e \in \Delta_{i,j})>X_{e_i}'~\Big|~X'_{e_i}\Big)\\
        &\hspace{0.5cm}=\left(1-(X'_{e_i})^2\right)^{|V(k,n)|-2}
    \end{align*}
    Using that $X'_{e_i}$ is uniformly distributed over $[0,\rW]$ and that $|V(k,n)|=n-k+1$, it follows that
    \begin{align*}
        \mathbb{P}\Big(e_{i}\in\rE\big(\mstr(\rK'_n[V(k,n)])\big)\Big)&\leq\frac{1}{\rW}\int_0^\rW(1-x^2)^{n-k+1}dx\\
        &\leq\frac{1}{\rW}\int_0^\infty e^{-(n-k+1)x^2}dx\\
        &=\frac{\sqrt{\pi}}{2\rW\sqrt{n-k+1}}\,,
    \end{align*}
    from which we obtain
    \begin{align*}
        \mathbb{P}(E')\geq1-\sum_{k<n^a}\sum_{i\in[m]}\frac{\sqrt{\pi}}{2\rW\sqrt{n-k-1}}\geq1-\frac{\sqrt{\pi}}{2}\frac{n^am}{\rW\sqrt{n-n^a-1}}\longrightarrow1\,,
    \end{align*}
    where the last convergence follows from $\rW=\frac{1}{\log n}$, $m=\lceil n^a\rceil+\rL-1 =\lceil n^a\rceil+\Lvalue-1$, and $a<\frac{1}{4}$. 
    
    Combining the fact that $\mathbb{P}(\rI<n^a)\to1$ with the definitions of $\rK^a_n$ and $\rU_\ell$, we now have that
    \begin{align}
        &\mathbb{P}\Big(\exists U\in\rU_\ell:\wdiam\big(\mstr(\rK_n[U])\big)>\epsilon\Big)\label{eq:left sets 1}\\
        &\hspace{0.5cm}=\mathbb{P}\Big(\exists U\in\rU_\ell:\wdiam\big(\mstr(\rK_n[U])\big)>\epsilon~\Big|~\rI<n^a\Big)+o(1)\notag\\
        &\hspace{0.5cm}\leq\mathbb{P}\Big(\exists k<n^a:\wdiam\big(\mstr(\rK^a_n[V(k,n)])\big)>\epsilon\Big)+o(1)\,,\notag
    \end{align}
    where the last inequality comes from the definition of $\rK_n^a$, and is due to $\rU_\ell=(V(\rI-1,n),\ldots,V(1,n))\subset(V(n^a-1,n),\ldots,V(1,n)))$ whenever $\rI<n^a$. Note that the coupling between $\rK^a_n$ and $\rK'_n$ only reduces the weight of the edges $e_1,\ldots,e_m$ in $\rK'_n$ relative to $\rK_n^a$, from which it follows that, if $e_i\notin\rE(\mstr(\rK'_n[V(k,n)]))$ for some $i\in[m]$, then $e_i\notin\rE(\mstr(\rK^a_n[V(k,n)]))$. This implies that, conditionally given $E'$, the trees $\mstr(\rK^a_n[V(k,n)])$ and $\mstr(\rK'_n[V(k,n)])$ are equal. Using that $\mathbb{P}(E')\to1$, we thus obtain
    \begin{align}
        &\mathbb{P}\Big(\exists k<n^a:\wdiam\big(\mstr(\rK^a_n[V(k,n)])\big)>\epsilon\Big)\label{eq:left sets 2}\\
        &\hspace{0.5cm}=\mathbb{P}\Big(\exists k<n^a:\wdiam\big(\mstr(\rK^a_n[V(k,n)])\big)>\epsilon~\Big|~E'\Big)+o(1)\notag\\
        &\hspace{0.5cm}=\mathbb{P}\Big(\exists k<n^a:\wdiam\big(\mstr(\rK'_n[V(k,n)])\big)>\epsilon~\Big|~E'\Big)+o(1)\,.\notag
    \end{align}
    Finally, consider a coupling between $\rK'_n$ and $\rK_n$ where $X'_e\leq X_e$ for any $e\in\rE(K_n)$ and such that $X'_e=X_e$ whenever $e\notin\{e_1,\ldots,e_m\}$. By using that $\mstr(\rK'_n)=\mstr(\rK_n)$ whenever $E'$ holds, it follows that
    \begin{align}
        &\mathbb{P}\Big(\exists k<n^a:\wdiam\big(\mstr(\rK'_n[V(k,n)])\big)>\epsilon~\Big|~E'\Big)\label{eq:left sets 3}\\
        &\hspace{0.5cm}=\mathbb{P}\Big(\exists k<n^a:\wdiam\big(\mstr(\rK_n[V(k,n)])\big)>\epsilon~\Big|~E'\Big)\notag\\
        &\hspace{0.5cm}=\mathbb{P}\Big(\exists k<n^a:\wdiam\big(\mstr(\rK_n[V(k,n)])\big)>\epsilon\Big) + o(1)\,,\notag
    \end{align}
    where we used that $\mathbb{P}(E')\to1$ for the last equality. Now, using Theorem~\ref{thm:bound on diam(mst)} similarly as before, we obtain that
    \begin{align*}
        \mathbb{P}\Big(\exists k<n^a:\wdiam\big(\mstr(\rK_n[V(k,n)])\big)>\epsilon\Big)\longrightarrow0\,.
    \end{align*}
    The proof of this lemma now follows by combining \eqref{eq:left sets 1}, \eqref{eq:left sets 2}, and \eqref{eq:left sets 3}.
\end{proof}

With the above lemmas in hand, the proof of Proposition~\ref{prop:good sets star path} is routine.

\begin{proof}[Proof of Proposition~\ref{prop:good sets star path}]
    Fix $\epsilon>0$ and let $\rW=\frac{1}{\log n}$ and $\rL=\lfloor\log\log n\rfloor$. Then
    \begin{align*}
        \mathbb{P}\Big(\exists U\in\rU:\wdiam\big(\mstr(\rK_n[U])\big)>\epsilon \Big)&=        \mathbb{P}\Big(\exists U\in\rU_r^-\cup\rU_r^+\cup\rU_\ell:\wdiam\big(\mstr(\rK_n[U])\big)>\epsilon \Big)\\
        &\leq\mathbb{P}\Big(\exists U\in\rU_r^-:\wdiam\big(\mstr(\rK_n[U])\big)>\epsilon \Big)\\
        &\hspace{0.5cm}+\mathbb{P}\Big(\exists U\in\rU_r^+:\wdiam\big(\mstr(\rK_n[U])\big)>\epsilon \Big)\\
        &\hspace{0.5cm}+\mathbb{P}\Big(\exists U\in\rU_\ell:\wdiam\big(\mstr(\rK_n[U])\big)>\epsilon \Big)\,,
    \end{align*}
    and the right hand side converges to $0$ by Lemma~\ref{lem:close right}, \ref{lem:far right}, and \ref{lem:left}, proving the proposition.
\end{proof}

\section{Conclusion}\label{sec:conclusion}

\subsection{More general weight distributions}

The extension of Theorem~\ref{thm:MAIN} from Uniform$[0,1]$ to more general weight distributions is quite straightforward. Fix a probability density function $f:[0,\infty) \to [0,\infty)$,
and let 
$\rho^*=\sup(x:\int_0^xf(y)dy<1)$.
Let 
$\rX'=(X_e',e \in \rE(K_n))$ be independent random variables with density $f$, and let 
$\rK_n'=(K_n,\rX')$. 
\begin{thm}\label{thm:MAIN2}
Suppose that $f(0)>0$, that $f$ is continuous at zero, and that $\rho^* < \infty$. 
Fix any sequence $(H_n,n \ge 1)$ of connected graphs with $H_n$ being a spanning subgraph of $K_n$. Then for any $\eps > 0$, as $n \to \infty$, 
\begin{itemize}
    \item[(a)] with high probability there exists an \mstr\ sequence $\rS$ for $(\rK_n',H_n)$ with $\wt(\rS) \le \rho^*+\eps$, and 
    \item[(b)] there exists $\delta > 0$ such that with high probability, given any optimizing sequence $\rS=(S_1,\ldots,S_m)$ for $(\rK_n',H_n)$ with $\wt(\rS) \le \rho^*-\eps$, the final spanning subgraph $H_{n,m}$ has weight $\mathop{w}(H_{n,m}) \ge \delta n  \mathop{w}(\mstr(\rK_n))$.
\end{itemize}
In particular, $\cost(\rK_n',H_n) \convp \rho^*$ as $n \to \infty$.
\end{thm}
The proof is very similar to that of Theorem~\ref{thm:MAIN}, so we only describe the changes that are required to prove the more general version. 

The proof of the lower bound, part (b), proceeds just as in the case of Uniform$[0,1]$ edge weights: for any $\eps > 0$, any optimizing sequence $\rS=(S_0,\ldots,S_m)$ for $(\rK_n',H_n)$ with $\wt(\rS) \le \rho^*-\eps$ leaves edges of weight greater than $\rho^*-\eps$ untouched, so all such edges appear in the final subgraph $H_{n,m}$. The number of such edges is Binomial$\big(|\rE(H_n)|,\int_{\rho^*-\eps}^{\rho^*}f(x)dx\big)$-distributed, so with high probability there are a linear number of such edges. On the other hand, $\mathop{w}(\mstr{\rK_n}) \to \zeta(3)/f(0)$ in probability \cite{frieze1985value}, and the lower bound follows.

For the upper bound, note that the bounds on the total cost of the optimizing sequences we construct essentially all have the form $A+B$ where $A$ is the greatest weight of a single edge, and $B$ is the weighted diameter of the minimum spanning tree of some subgraph of $K_n$. In order to prove  Theorem~\ref{thm:MAIN}, we used that $A \le 1$, and proved using  Theorem~\ref{thm:bound on diam(mst)} and Proposition~\ref{prop:good sets star path} that we could take $B$ as close to zero as we wished (by a careful choice of optimizing sequence). For the edge weights $\rX'$, we can simply replace the bound $A \le 1$ by the bound $A \le \rho^*$. To show that we can make $B$ as close to zero as we like, 
we can carry through the same proof as in the Uniform$[0,1]$ case, provided that versions of  Theorem~\ref{thm:bound on diam(mst)} and Proposition~\ref{prop:good sets star path} are still available to us.
 
To see that  Theorem~\ref{thm:bound on diam(mst)} and Proposition~\ref{prop:good sets star path} do essentially carry over to the setting of $\rK_n'=(K_n,\rX')$, 
we make use of the following coupling. 
For $t \in [0,\rho^*]$ let $g(t)=\mathbb{P}(X' \le t)$, so that $g(X')$ is Uniform$[0,1]$-distributed. We can thus couple the random weights $\rX'$ to independent Uniform$[0,1]$ weights $\rX=(X_e,e \in \rE(K_n))$ by taking $X_e=g(X_e')$, and thereby couple $\rK_n'=(K_n,\rX')$ to $\rK_n=(K_n,\rX)$. The edge weights $\rX'=(X'_e,e \in \rE(K_n))$ are almost surely pairwise distinct, and on this event, 
the ordering of $\rE(K_n)$ in increasing order of weight is the same for the weights $\rX$ and $\rX'$ and thus $\mstr(\rK'_n)=\mstr(\rK_n)$. 

Since $f(0)>0$ and $f$ is continuous, for all $u$ sufficiently small we have $f(u)>f(0)/2$ and $g(u)\ge uf(0)/2$. It follows in particular that if $X_e \le u f(0)/2$
then $X_e'\le 2X_e/f(0) \le u$. 
This observation implies that, under the above coupling between $\rK_n$ and $\rK_n'$, if $\wdiam(\mstr(\rK_n)) \le uf(0)/2$ then
$\wdiam(\mstr(\rK_n')) \le u$, and Theorem~\ref{thm:bound on diam(mst)} thus yields that for all $n$ sufficiently large, 
\begin{equation}\label{eq:bound on diam(mst)_replacement}
\mathbb{P}\left(\wdiam(\mstr(\rK')) \ge \frac{2}{f(0)}\frac{7 \log^4 n}{n^{1/10}}\right) \le 
\mathbb{P}\left(\wdiam(\mstr(\rK')) \ge \frac{7\log^4 n}{n^{1/10}}\right)
\le
\frac{4}{n^{\log n}}\, .
\end{equation}
Similarly, Proposition~\ref{prop:good sets star path} implies that (in the notation of that proposition), for all $\eps > 0$
    \[
    \mathbb{P}\Big(\exists U\in\rU:\wdiam\big(\mstr(\rK_n'[U])\big)>2\epsilon/f(0) \Big)\longrightarrow 0\,
\]
as $n \to \infty$. But since $\eps > 0$ was arbitrary, this implies that also 
\begin{equation}\label{eq:good sets star path_replacement}
\mathbb{P}\Big(\exists U\in\rU:\wdiam\big(\mstr(\rK_n'[U])\big)>\epsilon \Big)\longrightarrow 0\, 
\end{equation}
 for all $\eps > 0$. 

All the remaining ingredients of the proof of Theorem~\ref{thm:MAIN} use only information about the graph-theoretic structure of $\mstr(\rK_n)$, not its weights, and so carry over to the setting of non-uniform weights (using the fact that $\mstr(\rK_n')$ and $\mstr(\rK_n)$ have the same distributions as unweighted graphs -- indeed, they are equal under the above coupling). By running the proof of Theorem~\ref{thm:MAIN} but replacing all expressions of the form $1+\wdiam(F)$ by $\rho^*+\wdiam(F)$, and when needed invoking 
\eqref{eq:bound on diam(mst)_replacement} and \eqref{eq:good sets star path_replacement} in place of Theorem~\ref{thm:bound on diam(mst)} and Proposition~\ref{prop:good sets star path}, respectively, we obtain Theorem~\ref{thm:MAIN2}.

\medskip
Before concluding this subsection, we note that if $\rho^*=\infty$ then for any $r > 0$, the probability that at least one edge of $H_n$ has weight at least $r$ tends to $1$, so $\mathbb{P}(\cost(\rK_n',H_n) > r) \to 1$ as $n \to \infty$. Thus, in this case we also have $\cost(\rK_n',H_n) \convp \rho^*$. 

\subsection{Open questions and future directions}
This work introduces the notion of local minimum spanning tree searches and proves a weak law of large numbers for the cost of such local searches.
The framework naturally suggests several directions for future research, some of which we now highlight.

\begin{itemize}
    \item Our main results concern low-weight \mstr\ sequences $\rS$ for randomly weighted complete graphs, where $\wt(\rS)$ is measured in the $L_\infty$ sense: it is the maximum weight of any single step of the optimizing sequence. However, one may wish to vary the norm used to  measure the weights of optimizing sequences. The other $L_p$ norms are natural alternatives, and correspond to studying the values 
    \begin{align*}
    \cost{}_p(\rG,H)=\min\Big\{\wt{}_p(\rS):\rS\textrm{ is an \mstr\  sequence for }(\rG,H)\Big\}
\end{align*}
    where 
    \begin{align*}
    \wt{}_p(\rS)=\left(\sum_{i=1}^m\big(\wt(\rS,i)\big)^p\right)^\frac{1}{p}
\end{align*}
is the $L^p$ norm of $(\wt(\rS,i),1\leq i\leq m)$. 

At first sight, using $L^1$ weights may seem very natural, as it corresponds to the total weight of all the subgraphs modified by the sequence. Mathematically, however, in the setting considered in this paper the $L^1$ cost is quite easy to understand. Indeed, 
for $\rK_n$ and $H_n$ as in Theorem~\ref{thm:MAIN}, by considering the sequence $\rS=([n])$ which simply replaces $H_n$ by $\mstr(\rK_n)$ in one step, we obtain that
\begin{align*}
    \cost{}_1(\rK_n,H_n)\leq w(H_n)\,.
\end{align*}
Conversely, since any edge of $e\in\rE(H_n)\setminus\rE(\mstr(\rK_n))$ must be removed in order to form the \mstr, for any $\mstr$ sequence $\rS=(S_1,\ldots,S_m)$, there must exist $i\in[m]$ such that $e\in\rE(H_{n,i-1}[S_i])$. This implies that
\begin{align*}
    \wt{}_1(\rS)\geq\sum_{i\in[m]}w\big(H_{n,i-1}[S_i]\big)\geq\sum_{e\in\rE(H_n)\setminus\rE(\mstr(\rK_n))}X_e=\big(1+o_\mathbb{P}(1)\big)w(H_n)\,,
\end{align*}
where the final asymptotic follows from the fact that $H_n$ is chosen independently of $\rX$ and that any fixed edge belongs to the \mstr\ with probability $(n-1)/\binom{n}{2}=o_\mathbb{P}(1)$. 
Since the lower bound $\sum_{e\in\rE(H_n)\setminus\rE(\mstr(\rK_n))}X_e$ does not depend on the choice of \mstr\ sequence $\mathbb{S}$, it is also a lower bound on $\cost_1(\rK_n,H_n)$, and thus 
\begin{align*}
    \frac{\cost{}_1(\rK_n,H_n)}{w(H_n)}\longrightarrow1
\end{align*}
in probability. When $p < 1$, this argument can be adapted to prove the same convergence result for $\cost{}_p(\rK_n,H_n)$. However, when $p > 1$ it is less clear what behaviour to expect, and in particular it is unclear whether the dependence on the initial spanning subgraph $H_n$ will play a more complicated role.
\item Another natural modification of the setting is to measure the cost of a step by the {\em size}, rather than the weight, of the subgraph which is replaced by its \mstr. That is, we may define 
 \begin{align*}
    \mathrm{wt}'(\rG,H,\rS):=\max\Big\{\big|\rE(H_{i-1}[S_i])\big|:1\leq i\leq m\Big\}\,,
\end{align*}
and study
\[
\mathrm{cost}'(\rG,H)=\min\Big\{\mathrm{wt}'(\rG,H,\rS):\rS\textrm{ is an \mstr\  sequence for }(\rG,H)\Big\}
\]
For this notion of cost, even the behaviour of $\cost{}'(\rK_n,K_n)$ is unclear to us; how $\cost{}'(\rK_n,H_n)$ will depend on the starting graph $H_n$ is likewise unclear. However, at a minimum we expect that $\cost{}'(\rK_n,H_n) \to \infty$ in probability, provided that the initial spanning subgraphs $H_n$ are chosen independently of the weights.
\item Our result proves the {\em existence} of \mstr\ sequences of weight at most $(\rho^*+\eps)$, with high probability. However, our construction does not yield insight into the ubiquity of such sequences, and it would be interesting to know whether low-weight \mstr\ sequences can be found easily and without using ``non-local'' information. For example, suppose that at each step we choose a subgraph to optimize uniformly at random over all subgraphs of weight at most $w$. For which values of $w$ will the resulting sequence be an \mstr\ sequence with high probability? 

\item What is the asymptotic behaviour of $\cost(\rK_n,H_n)-\rho^*$? In particular, is there a sequence $a_n$ such that $a_n(\cost(\rK_n,H_n)-\rho^*)$ converges in distribution to a non-trivial random variable? 

\item What happens if $(K_n,\rX_n)$ is replaced by a different fixed connected, weighted graph $\rG_n=(G_n,\rX_n)$? How does the asymptotic behaviour of $\cost(\rG_n,H_n)$ depend on $G_n$? 

\item What happens if the iid structure of the edge weights of $\rK_n$ is modified? For example, one might generate $\rX_n$ by first taking $n$ independent, uniformly random points $P_1,\ldots,P_n \in [0,1]^d$, then letting $X_{ij} = |P_i-P_j|$ be the Euclidean distance between $i$ and $j$.

\end{itemize}

\appendix

\section{Bounds on the weighted diameter} \label{app:diam_bd_proof}

In this section, we prove Theorem~\ref{thm:bound on diam(mst)}.
The proof exploits Kruskal's algorithm for constructing minimum spanning trees. We first recall a very useful connection between Kruskal's algorithm run on the complete graph with independent Uniform$[0,1]$ edge weights $\rX=(X_e,e \in E(K_n))$ and the Erd\H{o}s-R\'enyi random graph process. In this setting, Kruskal's algorithm may be phrased as follows. Write $N=\binom{n}{2}$
\begin{itemize}
    \item Order the edges of $\rE(\rK_n)$ in incresasing order of weight as $e_1,\ldots,e_{N}$. 
    \item Let $F_0=([n],\emptyset)$ be the forest with vertex set $n$ and no edges.
    \item For $1 \le i \le {N}$, if $e_i$ joins distinct connected components of $F_{i-1}$ then let $\rE(F_i)= \rE(F_{i-1})\cup\{e_i\}$; otherwise let $F_i=F_{i-1}$.
\end{itemize}
The final forest $F_{N}$ is  $\mstr(\rK_n)$. 

\medskip

The Erd\H{o}s-R\'enyi random graph process can be described very similarly: 
\begin{itemize}
    \item Order the edges of $\rE(\rK_n)$ in increasing order of weight as $e_1,\ldots,e_{N}$.
    \item Let $G_0=([n],\emptyset)$ be the graph with vertex set $n$ and no edges.
    \item For $1 \le i \le {N}$, let $\rE(G_i)= \rE(G_{i-1})\cup\{e_i\}$. 
\end{itemize}
It is straightforward to see by induction that $F_i$ and $G_i$ always have the same connected components and, more strongly, that $F_i$ is the minimum spanning forest of $G_i$ (in that each tree of $F_i$ is the minimum spanning tree of the corresponding connected component of $G_i$). 

We also take $G(n,p)$ to be the subgraph of $\rK_n$ with edge set $\{e\in\rE(\rK_n): X_e \le p\}$. Since we ordered the edges in increasing order of weight as $e_1,\ldots,e_{N}$, the edge set of $G(n,p)$ is thus $\{e_1,\ldots,e_m\}$, where $m=m(p)$ is maximal so that $X_{e_m} \le p$. We likewise let $F(n,p)$ be the subgraph of $F_{N}$ consisting of all edges of $F_{N}$ with weight at most $p$, and note that $F(n,p)=F_{m(p)}$.

With this coupling in hand, we next explain our approach to bounding the weighted diameter of $\mstr(\rK_n)$. Our bound has two parts. Fix $p \in (0,1)$, and let $T^{\max}_{n,p}$ be the largest connected component of $F(n,p)$, with ties broken lexicographically. Note that $T^{\max}_{n,p}$ is a subgraph of $\mstr(\rK_n)$. Further write $L_{n,p}$ for the greatest number of edges in any path of $\mstr(\rK_n)$ which has exactly one vertex lying in $T^{\max}_{n,p}$. Finally,  write $W_n$ for the greatest weight of any edge of $\mstr(\rK_n)$. 
\begin{prop}\label{prop:mst_upper}
For any $p \in (0,1)$, 
\begin{equation*}
    \wdiam\big(\mstr(\rK_n)\big) \le 
    p\big(|T^{\max}_{n,p}|-1\big) + 2W_n L_{n,p}\, .
\end{equation*}
\end{prop}
\begin{proof} 
Fix any path $P$ in $\mstr(\rK_n)$. Then the set of vertices of $P$ contained in $T^{\max}_{n,p}$ form a subpath of $P$, since otherwise $\mstr(\rK_n)$ would contain a cycle; call this subpath $P_0$. Then $P_0$ contains at most $|T^{\max}_{n,p}|$ vertices, so at most $|T^{\max}_{n,p}|-1$ edges, and each such edge has weight at most $p$. Moreover, the edges of $P$ not lying in $P_0$ form at most two subpaths of $P$. Each of these subpaths has at most $L_{n,p}$ edges, so the number of edges of $P$ which are not edges of $P_0$ is at most $2L_{n,p}$; and the edges of $P$ which are not edges of $P_0$ all have weight at most $W_n$.
\end{proof}
To exploit this bound and prove Theorem~\ref{thm:bound on diam(mst)}, we must bound $|T^{\max}_{n,p}|$ and $L_{n,p}$, for some well chosen value of $p$, and bound $W_n$. The latter bound is the easiest, and we take care of it first. We will need the following bound on the probability of connectedness of $G(n,p)$. We believe we have seen this bound in the literature, but were unable to find a reference, so we have included its short proof. 

\begin{lemma}\label{lem:Gnp connectivity}
Let $G \sim G(n,p)$. Then 
\[
\mathbb{P}\left( G \text{ is not connected} \right) \leq e^{ne^{-\frac{np}{2}}} - 1
\]
\end{lemma}

\begin{proof} 
    Let $S$ be a subset of $[n]$ such that $S\neq\emptyset$ and $S\neq[n]$. Then
    \begin{align*}
        \mathbb{P}\Big(\textrm{$S$ is not connected to $S^c$ in $G$}\Big)=(1-p)^{|S|(n-|S|)} \,.
    \end{align*}
    This implies that
    \begin{align*}
        \mathbb{P}\big(\textrm{$G$ is not connected}\big)&=\mathbb{P}\big(\exists S\subseteq[n]:1\leq|S|\leq n/2\textrm{ and $S$ is not connected to $S^c$ in $G$}\big)\\
        &\leq\sum_{S\subseteq[n]:1\leq|S|\leq n/2}\mathbb{P}\Big(\textrm{$S$ is not connected to $S^c$ in $G$}\Big)\,.
    \end{align*}
    Combined with the previous result, this leads to
    \begin{align*}
        \mathbb{P}\big(\textrm{$G$ is not connected}\big)&=\sum_{S\subseteq[n]:1\leq|S|\leq n/2}(1-p)^{|S|(n-|S|)}\leq\sum_{1\leq k\leq n/2}\binom{n}{k}(1-p)^{k(n-k)}\,.
    \end{align*}
    Use now that $(n-k)\geq n/2$ along with the fact that $1-p\geq0$ to obtain that
    \begin{align*}
        \mathbb{P}\big(\textrm{$G$ is not connected}\big)&\leq\sum_{1\leq k\leq n}\binom{n}{k}(1-p)^{kn/2}=\big(1+(1-p)^{n/2}\big)^n-1\,.
    \end{align*}
    Finally, by using twice the convexity of exponential, we have
    \begin{align*}
        \mathbb{P}\big(\textrm{$G$ is not connected}\big)&\leq \left(1+e^{-\frac{pn}{2}}\right)^n-1\leq e^{ne^{-\frac{pn}{2}}} - 1 \,,
    \end{align*}
    which is the desired result.
\end{proof}

\begin{fact}\label{fact:wn_upper}
For all $n$ sufficiently large, it holds that $\mathbb{P}(W_n > 3\log^2 n/n) \le 1/n^{\log n}$.
\end{fact}
\begin{proof}
Under the above coupling, $F(n,p)$ and $G(n,p)$ have the same connected components, so 
\[
\mathbb{P}\big(W_n > 3\log^2 n/n\big)
=\mathbb{P}\Big(F(n,3\log^2 n/n)\mbox{ is not connected}\Big)
= \mathbb{P}\Big(G(n,3\log^2 n/n)\mbox{ is not connected}\Big). 
\]
Use now the bound from Lemma~\ref{lem:Gnp connectivity} to obtain that
\[
\mathbb{P}\Big(G(n,3\log^2 n/n)\mbox{ is not connected}\Big) 
\le \exp(ne^{-(3/2)\log^2 n})-1 = e^{n/(n^{\log n})^{3/2}}-1 \leq 1/n^{\log n}\, 
\]
the final bound holding for all $n$ sufficiently large.
\end{proof}

\begin{proof}[Proof of Theorem~\ref{thm:bound on diam(mst)}]
We prove the theorem by bounding $|T^{\max}_{n,p}|$ and $L_{n,p}$, for a carefully chosen value of $p$ (spoiler: we will take $p=1/n + 1/n^{11/10}$), then applying Proposition~\ref{prop:mst_upper}. Our arguments lean heavily on results from \cite{addario2006diameter}, and we next introduce those results (and the terminology necessary to do so). 

For $c > 0$, let $\mathop{\alpha}(c)$ be the largest real solution of $e^{-cx}=1-x$ (the quantity $\mathop{\alpha}(c)$ is the survival probability of a Poisson$(c)$ branching process). The key to the proof is the fact that the size of the largest component of $G(n,p)$ is with high probability close to $n\mathop{\alpha}(np)$ when $p=(1+o(1))/n$. We now provide a precise and quantitative version of this statement, with error bounds. 

By~\cite[Exercise~21~(d)]{addario2015partition}, for $\eps \ge 0$ we have 
\[
2\eps(1-o(1)) \le \mathop{\alpha}(1+\eps) \le 2\eps\, ,
\]
the first inequality holding as $\eps \to 0$. In particular, 
\begin{equation}
(3/2)\eps \le \mathop{\alpha}(1+\eps) \le 2\eps      
\end{equation}
for all $\eps \ge 0$ sufficiently small.

For the remainder of the proof, fix $p=1/n+1/n^{11/10}$ and write $s^+= n\mathop{\alpha}(n\log(1/(1-p))+ n^{3/4}$ and $s^-= n\mathop{\alpha}(n\log(1/(1-p))- 2n^{3/4}$. (Aside: for the careful reader who is verifying the connections to the results from \cite{addario2015partition}, note that $s^+=t^+$ but $s^-\neq t^-$, where $t^+,t^-$ are defined in~\cite[Proof of Theorem~4.4, Case~2]{addario2015partition}). By~\cite[Exercise~23~(a)]{addario2015partition}, for all $n$ sufficiently large we have 
\[
n \mathop{\alpha}(np) \le n\mathop{\alpha}(n\log(1/(1-p)) \le n \mathop{\alpha}(np)+\frac{2n^{1/2}}{1-p}, 
\]
and using the above bounds on $\mathop{\alpha}$, this yields
\[
n^{9/10} \le s^- \leq s^+ \le 3 n^{9/10}\,,
\]
for $n$ sufficiently large.

Let $\mathcal{C}^{\max}$ be the largest connected component of $G(n,p)$, and let $\mathcal{C}^{\mathrm{runnerup}}$ be its second largest component. Using the previous inequality on $s^-$ and $s^+$, by~\cite[(4.7)]{addario2015partition} we have 
\begin{equation}\label{eq:cmax_upper}
\mathbb{P}\big(|\mathcal{C}^{\max}|\ge 3n^{9/10}\big) 
\le 
\mathbb{P}\big(|\mathcal{C}^{\max}| \ge s^+\big) \le ne^{-(25/2)n^{1/10}} \, ;
\end{equation}
moreover, by~\cite[(4.10)]{addario2015partition}, we have 
\begin{equation}\label{eq:cmax_lower}
\mathbb{P}\big(|\mathcal{C}^{\max}| \le n^{9/10}\big)
\le \mathbb{P}\big(|\mathcal{C}^{\max}| \le s^-\big) \le 2ne^{-(25/2)n^{1/10}}\, ;
\end{equation}
finally, by~\cite[(4.10) and(4.11)]{addario2015partition}, we have 
\begin{equation}\label{eq:c2_upper}
    \mathbb{P}\big(|\mathcal{C}^{\mathrm{runnerup}}| \ge n^{4/5}\big) \le 5ne^{-(25/2)n^{1/10}}\, .
\end{equation}
Furthermore, under the coupling between $G(n,p)$ and $F(n,p)$, we have $|\mathcal{C}^{\max}| = |T^{\max}_{n,p}|$, so \eqref{eq:cmax_upper} immediately gives us that for all $n$ sufficiently large, 
\begin{equation}\label{eq:tnpmax_upper}
    \mathbb{P}\big(|T^{\max}_{n,p}| \ge 3n^{9/10}\big) \le ne^{-(25/2)n^{1/10}}\,. 
\end{equation}

It remains to bound $L_{n,p}$. For this, we use \eqref{eq:cmax_lower} and a Prim's-algorithm-type construction to control the greatest number of connected components of $G(n,p)$ that any path of $\mstr(\rK_n)$ lying outside $T^{\max}_{n,p}$ passes through, and use 
\eqref{eq:c2_upper} to bound the size of those components.

Condition on the graph $G(n,p)$, and fix a connected component $\mathcal{C}_1$ of $G(n,p)$ different from $\mathcal{C}^{\max}$. Let $f_1=u_1v_1$ be the smallest-weight edge with exactly one endpoint in $\mathcal{C}_1$, and let $p_1$ be its weight. Then $p_1 > p$, and $f_1$ is a cut-edge of $G(n,p_1)$. It follows that $f_1$ is an edge of $\mstr(\rK_n)$. Moreover, by the exchangeability of the edge weights, the endpoint $v_1$ of $f_1$ not lying in $\mathcal{C}_1$ is uniformly distributed over the remainder of the vertices, so 
\[
\mathbb{P}\Big(v_1 \not\in \mathcal{C}^{\max}~\Big|~G(n,p)\Big) \le 1-\frac{|\mathcal{C}^{\max}|}{n-|\mathcal{C}_1|} < \, 
1-\frac{|\mathcal{C}^{\max}|}{n}.
\]

If $v_1$ is not in $\mathcal{C}^{\max}$, then it lies in another connected component $\mathcal{C}_2$. Let $f_2=u_2v_2$ be the smallest-weight edge leaving $\mathcal{C}_1\cup \mathcal{C}_2$, and let $p_2$ be its weight. Then $f_2$ is an edge of $\mstr(\rK_n)$; to see this, note that any path $\gamma$ connecting $u_2$ and $v_2$ which is not just the edge $f_2$ contains some edge $e$ of weight strictly greater than $p_2$, meaning that $f_2$ is never the heaviest edge of any cycle. Moreover, the endpoint $v_2$ of $f_2$ not lying in $\mathcal{C}_1\cup\mathcal{C}_2$ is uniformly distributed over the remainder of the graph, so once again
\[
\mathbb{P}\Big(v_2 \not\in \mathcal{C}^{\max}~\Big|~G(n,p),v_1\not\in\mathcal{C}^{\max}\Big)
< 
1-\frac{|\mathcal{C}^{\max}|}{n}. 
\]
Continuing this process, we construct a sequence $\mathcal{C}_1,\ldots,\mathcal{C}_K$ of distinct connected components of $G(n,p)$ and a sequence $f_1,\ldots,f_K$ of edges of $\mstr(\rK_n)$, where where $f_i = u_iv_i$ is the smallest-weight edge from $\mathcal{C}_1\cup \ldots\cup \mathcal{C}_i$ to the remainder of the graph, $\mathcal{C}_1,\ldots,\mathcal{C}_K$ are all connected components of $G(n,p)$ different from $\mathcal{C}^{\max}$, and $v_K \in \mathcal{C}^{\max}$. To bound the length $K$ of the sequences, we use that at each step of the construction, the conditional probability that $f_j=u_iv_j$ has an endpoint in $\mathcal{C}^{\max}$ given $G(n,p)$ and given that $e_1,\ldots,e_{i-1}$ do not have an endpoint in $\mathcal{C}^{\max}$, is greater than $|\mathcal{C}^{\max}|/n$, and 
so 
\begin{align*}
\mathbb{P}\Big(K > k~\Big|~G(n,p)\Big) &= 
\mathbb{P}\Big(v_k \not\in\mathcal{C}^{\max}~\Big|~G(n,p)\Big) \\
& = 
\prod_{i=1}^k
\mathbb{P}\Big(v_i \not\in\mathcal{C}^{\max}~\Big|~G(n,p), v_1,\ldots,v_{i-1} \not \in \mathcal{C}^{\max}\Big)\\
& \leq \left(1-\frac{|\mathcal{C}^{\max}|}{n}\right)^k
\end{align*}

Note now that any path in $\mstr(\rK_n)$ with one endpoint in $\mathcal{C}_1$ and the other endpoint in $\mathcal{C}^{\max}$ passes through $\mathcal{C}_1,\ldots,\mathcal{C}_K$ and edges $f_1,\ldots,f_K$. Since each of the components $\mathcal{C}_1,\ldots,\mathcal{C}_K$ has size at most that of $\mathcal{C}^{\mathrm{runnerup}}$, it follows that the greatest number of edges in any path with one endpoint in $\mathcal{C}_1$ which only intersects $\mathcal{C}^{\max}$ in one vertex is at most $K |\mathcal{C}^{\mathrm{runnerup}}|$. Taking a union bound over the possible choices for $\mathcal{C}_1$ among all components of $G(n,p)$ different from $\mathcal{C}^{\max}$ (there are less than $n$ of them), it follows that 
\begin{align*}
    \mathbb{P}\Big(L_{n,p} > k\big|\mathcal{C}^{\mathrm{runnerup}}\big|~\Big|~ G(n,p)\Big)
    & \le n \mathbb{P}\Big(K > k~\Big|~ G(n,p)\Big) \\
    & \leq n\left(1-\frac{|\mathcal{C}^{\max}|}{n}\right)^k. 
\end{align*}
Recall now the tail bounds for $|\mathcal{C}^{\max}|$ and $|\mathcal{C}^{\mathrm{runnerup}}|$ from \eqref{eq:cmax_lower} and \eqref{eq:c2_upper} and use that
\begin{align*}
    &\mathbb{P}\Big(L_{n,p} > n^{9/10}\log^2 n~\Big|~ G(n,p),|\mathcal{C}^{\max}|>n^{9/10},|\mathcal{C}^{\mathrm{runnerup}}|<n^{4/5}\Big)\\
    &\hspace{0.5cm}=\mathbb{P}\Big(L_{n,p} > (n^{1/10}\log^2 n)\cdot n^{4/5}~\Big|~G(n,p),|\mathcal{C}^{\max}|>n^{9/10},|\mathcal{C}^{\mathrm{runnerup}}|<n^{4/5}\Big)\\
    &\hspace{0.5cm}\leq n\left(1-\frac{n^{9/10}}{n}\right)^{n^{1/10}\log^2n}\leq ne^{-\log^2n}
\end{align*}
to obtain
\begin{align}\label{eq:lnp_bound}
    \mathbb{P}\Big(L_{n,p} > n^{9/10}\log^2 n\Big)
    & \le 7ne^{-(25/2)n^{1/10}} + ne^{-\log^2n} \leq \frac{2}{n^{\log n}}\, ,
\end{align}
the last bound holding for $n$ large enough. 

We can now conclude the proof of Theorem~\ref{thm:bound on diam(mst)}.
By Fact~\ref{fact:wn_upper}, for $n$ sufficiently large, $\mathbb{P}(W_n > 3\log^2 n/n) \le 1/n^{\log n}$. Combined with \eqref{eq:lnp_bound}, this implies that 
\[
\mathbb{P}\left(2W_nL_{n,p} > \frac{6\log^4 n}{n^{1/10}}\right) \leq \frac{3}{n^{\log n}}\, .
\]
Using the bound of Proposition~\ref{prop:mst_upper} and combining it with the previous inequality and \eqref{eq:tnpmax_upper}, we obtain that
\begin{align*}
    \mathbb{P}\left(\wdiam\big(\mstr(\rK_n)\big) > 3pn^{9/10}+\frac{6\log^4 n}{n^{1/10}}\right) &\leq\mathbb{P}\left(\big|T^{\max}_{n,p}\big|\geq 3n^{9/10}\right) + \mathbb{P}\left(2W_nL_{n,p} > \frac{6\log^4 n}{n^{1/10}}\right) \\
    &\le ne^{-(25/2)n^{1/10}} + \frac{3}{n^{\log n}} \leq \frac{4}{n^{\log n}}\,,
\end{align*}
the last inequality holding when $n$ is large. 
Finally, since $p=1/n+1/n^{11/10}$, for $n$ large we have $3pn^{9/10}+\tfrac{6\log^4 n}{n^{1/10}}< \tfrac{7\log^4 n}{n^{1/10}}$, so the bound of Theorem~\ref{thm:bound on diam(mst)} follows.

\end{proof}

\section*{Acknowledgements}
This work was initiated at the Bellairs Research Institute of McGill University. The authors thank G\'abor Lugosi, who initially proposed a version of the problem studied in this work, at a Bellairs workshop on probability and combinatorics.

\bibliographystyle{siam}
\bibliography{articles}

\end{document}